\documentclass[preprint,12pt]{elsarticle}
\UseRawInputEncoding
\usepackage{xcolor}
\usepackage{amssymb}
\usepackage{amscd,amssymb,latexsym,srcltx}
\usepackage{dsfont}
\usepackage{amsthm}
\usepackage{setspace}
\usepackage{mdwlist}
\usepackage{enumerate}
\usepackage{graphics}
\usepackage{graphicx}
\usepackage{epsfig}
\usepackage{subfigure}
\usepackage{amssymb}
\usepackage{mathrsfs}
\usepackage{amsmath}
\allowdisplaybreaks
\usepackage{amsfonts}
\usepackage{geometry}
 \usepackage{hyperref}
\usepackage{geometry}
\usepackage{amsfonts}
\usepackage{mathrsfs}
\usepackage{amsmath}
\usepackage{amssymb}
\usepackage{amsthm}
\usepackage{geometry}
\biboptions{numbers,sort&compress} %可以压缩参考文献标号，还可以进行排序，即无论正文里面的顺序怎样，显示出来都是先后顺序。

\newtheorem{theorem}{Theorem}[section]
\newtheorem{assumption}[theorem]{Assumption}
\newtheorem{corollary}[theorem]{Corollary}

\newtheorem{definition}[theorem]{Definition}
\newtheorem{lemma}[theorem]{Lemma}
\newtheorem{remark}[theorem]{Remark}
\numberwithin{equation}{section}

\newcounter{cnum}
\makeatletter
\def\C{\@ifnextchar[{\@with}{\@without}}
\def\@with[#1]{\@ifundefined{c@#1}{%
    	   \newcounter{#1}%
            \@bsphack\protected@write\@auxout{}{
            \string\newlabel{#1}{{\thecnum} {\thepage}}
            \@esphack
            }
            C_{\thecnum}
            \stepcounter{cnum}%
        }{C_{\ref{#1}}}}
\def\@without{C_{\thecnum}
        \stepcounter{cnum}}
\makeatother
\journal{}

\begin{document}

\begin{frontmatter}

\title{Existence of compact $\varphi$-attracting sets and estimate of their attractive velocity for infinite-dimensional dynamical systems\tnoteref{t1}}
\tnotetext[t1]{The work is supported by National Natural Science Foundation of China (No.11731005; No.11801071).}
%% use optional labels to link authors explicitly to addresses:
%% \author[label1,label2]{<author name>}
%% \address[label1]{<address>}
%% \address[label2]{<address>}

\author[1,2]{Chunyan Zhao}

\author[2]{Chengkui Zhong\corref{cor1}}
\author[2]{Xiangming Zhu}

\address[1]{School of Mathematics and Big Data, Anhui University of Science and Technology, Huainan, 232001, PR China}
\address[2]{Department of Mathematics, Nanjing University, Nanjing, 210093, PR China}

 \cortext[cor1]{Corresponding author.\\
 E-mails: emmanuelz@163.com(C. Zhao), ckzhong@nju.edu.cn(C. Zhong), zxmxygmt@163.com(X. Zhu)}

\begin{abstract}
This paper is devoted to the quantitative study of the attractive velocity of   compact semi-invariant attracting sets for  infinite-dimensional dynamical systems.  We  introduce the notion of compact $\varphi$-attracting set whose attractive speed is characterized by a general non-negative decay function~$\varphi$,  and prove that~$\varphi$-decay with respect to  noncompactness measure is a sufficient condition for a dissipative system to have a compact $\varphi$-attracting set.  Furthermore, several criteria for~$\varphi$-decay with respect to noncompactness measure are provided.  Finally, as an application, we establish the existence of a compact exponential attracting set and the specific estimate of its attractive velocity for a semilinear wave equation with a critical nonlinearity.
\end{abstract}


\end{frontmatter}

\section{Introduction}
In this paper, we will study the attractive velocity of  compact semi-invariant attracting sets for infinite-dimensional dynamical systems.

The global attractor~$\mathcal{A}$ of a dynamical system~$(X, \{S(t)\}_{t\geq0})$~is a compact, invariant subset of~$X$~which attracts any bounded set~$B\subseteq X$. By the definition, the global attractor (if exists) captures all the asymptotic behavior of the system. According to the~H\"{o}lder-Ma\~{n}\'{e} theorem (see \cite{Foias1996,MR1710097}), each
compact set with finite fractal dimension is homeomorphic to a compact subset of Euclidean space~$\mathds{R}^n$. Therefore, if the fractal dimension of the global attractor is finite, then the infinite-dimensional dynamical system restricted to the global attractor is reducible to a finite-dimensional dynamical system (see \cite{MR2508165}).

Even so, as an object to describe the long-term behavior of dissipative systems, the global attractor still has its limitations. First, its definition does not contain quantitative information about the speed at which it attracts  bounded sets. In fact, it may attract  bounded sets at arbitrarily low rates. Second, the global attractor may be sensitive to perturbations. In general, the global attractor is upper semicontinuous with respect to perturbations, whereas its lower semicontinuity is much more difficult to obtain. Finally, when the system has a global attractor
with finite fractal dimension, the reduced finite-dimensional dynamical system given by the H\"{o}lder-Ma\~{n}\'{e} theorem is only H\"{o}lder continuous, not necessarily Lipschitz continuous, so it is not necessarily generated by ordinary differential equations (see \cite{MR2508165,Robinson2001}).

In order to overcome these limitations, Foias, Sell and Temam \cite{Inertialmanifolds} proposed the notion of inertial manifold in 1988. It is defined as a positively invariant, finite-dimensional, Lipschitz manifold which exponentially attracts bounded subsets and contains the global attractor. The dynamical system restricted to the inertial manifold is reducible to a Lipschitz continuous system of ordinary differential equations which is called the inertial form of the original system. Almost all known methods of constructing inertial manifolds are based on the ``spectral interval" condition (see \cite{Inertialmanifolds}), which is difficult to verify. The existence of inertial manifolds has been proved for a large number of equations with space dimension  1 or 2 (see \cite{MR966192,Inertialmanifolds,Robinson2001,MR1873467,TEMAM}). However, its existence is still an open
problem for several important equations, such as the two-dimensional incompressible Navier-Stokes equations. Furthermore, its nonexistence  has been proved for damped Sine-Gordon equations \cite{MR921912}.

As it is not always possible to obtain inertial manifolds, Eden, Foias, Nicolaenko and Temam \cite{MR1335230} proposed the concept of exponential attractor in 1994. An exponential attractor is a positively invariant, compact set which has finite fractal dimension  and attracts all bounded subsets at exponential rates. Actually, to the best of our knowledge, exponential attractors exist indeed for almost all equations with finite-dimensional global attractors.

If a dynamical system possesses an exponential attractor~$A_{exp}$, then $\mathcal{A}=\omega(A_{exp})\subseteq A_{exp}$  is its global attractor, where~$\omega(A_{exp})$ denotes the~$\omega$-limit set of $A_{exp}$. Furthermore, its global attractor $\mathcal{A}=\omega(A_{exp})$ has a finite fractal dimension.
The finiteness of the fractal dimension of the attractor is an ideal property that makes infinite dynamical systems   reducible to finite-dimensional systems. There are indeed a large number of reality-based models with finite-dimensional attractors.
However, it has been proved that the solution semigroups of many evolution equations have global attractors  with infinite fractal dimension; for instance, the~$p$-Laplace equations with $p>2$ (see \cite{MR2728546}), some reaction-diffusion equations in unbounded domains (see \cite{MR1871475,Zelik2003}), some hyperbolic equations in unbounded domains (see \cite{Zelik2001}), and so on. When a dynamical system has a global attractor of infinite fractal dimension, it has no exponential attractors, but may still have positively invariant and exponentially attractive compact sets. Based on this observation, Zhang, Kloeden, Yang and Zhong \cite{Zhangjin1} thought that the properties of exponential attractiveness and finite fractal dimension should be discussed separately, and proposed the concept of exponential decay with respect to noncompactness measure for the first time. They have proved that a sufficient and necessary condition for a dissipative dynamical system to have a positively invariant and exponentially attractive compact  set~$\mathcal{A}^*$~is that the noncompactness measure of bounded sets decays exponentially. They also gave some criteria for exponential decay with respect to noncompactness measure and proved this property for a class of wave equations with weak damping via the~$(C^{*})$~condition.

Inspired by the work of Zhang et al. \cite{Zhangjin1}, we will focus on the quantitative study of the attractive velocity of  compact semi-invariant attracting sets for dynamical systems in this paper. We notice that, when the damping is degenerate, there is no conclusion as to whether the fractal dimension of the global attractor is finite and whether the noncompactness measure decays exponentially for some dissipative systems.  Motivated by this, we put forward the concept of compact $\varphi$-attracting set (see Definition~$\ref{def21-1-31-1}$) whose attractive speed is characterized by a general non-negative decay function~$\varphi$. In particular, when~$\varphi$ is an exponential function, it is a compact exponential attracting set, whose fractal dimension is not required to be finite compared with the exponential attractor; and when~$\varphi$ is a polynomial function, it is a compact polynomial attracting set. Following the idea in \cite{Zhangjin1}, we study the specific attractive speed of compact semi-invariant attracting sets by estimating the decay speed of noncompactness measure. To this end, we propose the notion of~$\varphi$-decay with respect to  noncompactness measure (see Definition~$\ref{def20-8-5-3}$), which is a sufficient condition for a dissipative system to have a compact $\varphi$-attracting set.  We also establish several criteria for~$\varphi$-decay with respect to noncompactness measure.

In the existing literature,  the research on the attractive velocity of  attractors  is mainly limited to the range of exponential attractive velocity. Besides, there are many studies on estimating  the decay velocity of energy functionals for various dynamical systems.
However, we have not found any literature on the quantitative study of the general attractive speed of compact semi-invariant attracting sets, and this article has filled this gap  to some extent.

On the other hand, all the criteria  for exponential decay of noncompactness measure  given by Zhang et al. \cite{Zhangjin1}  depend on the compactness of a system.  In the critical case, due to the lack of compactness,  these criteria are no longer applicable.
Chueshov \cite{Chueshov2008} proposed the method of contractive function by which the existence of the global attractor can be established for the critical case. We are inspired to  establish a criterion based on contractive function for~$\varphi$-decay with respect to noncompactness measure. This criterion  only involves some rather weak compactness associated with the repeated limit inferior and requires no compactness,  which makes it  suitable for critical cases. In addition, we prove that quasi-stable systems are  exponentially decaying with respect to noncompactness measure and therefore have  compact exponential attracting sets.

 Finally, as an application of the abstract theory, we establish the existence of a compact exponential attracting set and the specific estimate of its attractive velocity for a semilinear wave equation with a critical nonlinearity. We have  proved in \cite{my1,my2} the global well-posedness, dissipativity and the existence of the global attractor for a semilinear wave equation with nonlocal weak damping and anti-damping

  \begin{equation}\label{21-8-30-5}
  u_{tt}-\Delta u+k||u_{t}||_{L^2(\Omega)}^p u_t+f(u)
=\displaystyle\int_{\Omega}K(x,y)u_{t}(y)dy+h(x)
\end{equation}
in both the subcritical and critical cases. We hope to advance the study of  this wave equation and estimate the specific attractive velocity of compact semi-invariant attracting sets under critical conditions. However, due to the degenerate  damping,  we can't get the exponential attractive velocity. So we add a linear damping term~$lu_t\ (l>0)$ (see~$(\ref{wave equa3})$  below) to this model. Then  by the criterion based on contractive function and energy estimate, we obtain the decay rate of  noncompactness measure and thus verify the existence of a compact exponential attracting set.  Furthermore,  we prove that the attractive rate of the compact exponential attracting set depends on the coefficient~$l$ of the linear damping term.

We would like to point out that after completing the present paper we have carried forward the work on this topic and  verified in \cite{my4,my5} the existence of compact polynomial attracting sets under subcritical and critical conditions  respectively for the wave equation $(\ref{21-8-30-5})$,  that is,  the case of $l=0$  in Eq. $(\ref{wave equa3})$.

The paper is organized as follows. Section 2 contains some preliminaries.
$\varphi$-decay with respect to noncompactness measure is proved to be a sufficient condition for the existence of compact $\varphi$-attracting sets in Section 3. Section 4 presents several criteria for $\varphi$-decay with respect to noncompactness measure. In Section 5, as an application,  we establish the existence of a compact exponential attracting set for a semilinear wave equation with a critical nonlinearity.

 Throughout this paper,  the symbols~$\hookrightarrow$~and~$\hookrightarrow\hookrightarrow$~stand for continuous embedding and compact embedding respectively.
The capital letter ``C" with a (possibly empty) set of subscripts will denote
a positive constant depending only on its subscripts and may vary from one
occurrence to another.

\section{Preliminaries}
We first briefly recall the definition of Kuratowski measure of noncompactness and its basic properties. For more details, we refer to \cite{K. Deimling,MR1153247}.

\begin{definition}\cite{K. Deimling,MR1153247}\label{def20-10-26-1}
 Let~$(X,d)$~be a metric space and let~$B$~be a bounded subset of~$X$.
The Kuratowski~$\alpha$-measure of noncompactness is defined by
\begin{equation*}
  \alpha(B) = \inf\{\delta>0 | B\   \text{has a finite cover of diameter} ~<\delta\}.
\end{equation*}
\end{definition}

\begin{lemma}\cite{K. Deimling,MR1153247}\label{lemma21-2-10-1}
Let~$(X,d)$~be a complete metric space and~$\alpha$~be the Kuratowski measure of noncompactness. Then
\begin{enumerate}[(i)]
  \item~$\alpha(B)=0 $~if and only if~$B$~is precompact;
  \item~$\alpha(A)\leq\alpha(B)$~whenever~$A\subseteq B$;
  \item~$\alpha(A \cup B)=\max\{\alpha(A), \alpha(B)\}$;
  \item~$\alpha(B)=\alpha(\overline{B})$, where~$\overline{B}$~is the closure of~$B$;
  \item~if~$B_1\supseteq B_2\supseteq B_3\ldots$~are nonempty closed sets in~$X$~such that~$\alpha(B_n)\rightarrow 0$~as~$n\rightarrow \infty$, then~$\cap_{n\geq1}B_n$~is nonempty and compact;
\item~if~$X$~is a Banach space, then~$\alpha(A+B)\leq \alpha(A)+\alpha(B)$.
\end{enumerate}
\end{lemma}

Next, we will briefly review the definitions and fundamental conclusions of dynamical systems and the global attractor.
\begin{definition}\cite{Robinson2001,TEMAM,Chueshov2008}
  A dynamical system is a pair of objects~$(X, \{S(t)\}_{t\geq0})$~consisting of a
complete metric space~$X$~and a family of continuous mappings~$\{S(t)\}_{t\geq0}$~of X
into itself with the semigroup properties:
\begin{enumerate}[(i)]
  \item ~$S(0)=I$,
  \item ~$S(t+s)=S(t)S(s)$~for all~$t,s\geq0$,
\end{enumerate}
where~$X$~is called a phase
space (or state space) and~$\{S(t)\}_{t\geq0}$~is called an evolution semigroup.
\end{definition}

\begin{definition}\cite{Robinson2001,TEMAM,Chueshov2008}
Let~$\{S(t)\}_{t\geq0}$~be a semigroup on a complete metric space~$(X,d)$. A closed set~$\mathcal{B}\subseteq X$~is said to be absorbing for~$\{S(t)\}_{t\geq0}$~iff for any bounded
set~$B\subseteq X$~there exists~$t_0(B)$~$($the entering time of~$B$~into~$\mathcal{B}$$)$ such that~$S(t)B\subseteq \mathcal{B}$~for all~$t\geq t_0(B)$.
 $\{S(t)\}_{t\geq0}$~is said to be dissipative iff it possesses a bounded absorbing set.
\end{definition}

\begin{lemma}\cite{Chueshov2008}\label{20-5-5-2}
Let~$\{S(t)\}_{t\geq0}$~be a semigroup on a complete metric space~$(X,d)$. If~$\{S(t)\}_{t\geq0}$ is dissipative, then it possesses a positively invariant bounded absorbing set. To be more precise, let~$ \mathcal{B}$~be its bounded absorbing set, then~$ \mathcal{B}_0=\overline{\bigcup\limits_{t\geq t_0(\mathcal{B})}S(t)\mathcal{B}}$~is a positively invariant bounded absorbing set, where~$t_0(\mathcal{B})>0$~is the entering time of~$ \mathcal{B}$~into itself.
\end{lemma}

\begin{definition} \cite{Robinson2001,TEMAM,Chueshov2008}\label{def20-11-11-1}
A compact set~$\mathcal{A}\subseteq X$~is said to be a global attractor of the dynamical system~$(X, \{S(t)\}_{t\geq0})$~iff
\begin{enumerate}[(i)]
  \item~$\mathcal{A}\subseteq X$~is an invariant set, i.e.,~$S(t)\mathcal{A}=\mathcal{A}$~for all~$t\geq0$,
  \item~$\mathcal{A}\subseteq X$~is uniformly attracting, i.e., for all bounded set~$B\subseteq X$~we have
  \begin{equation*}
\lim_{t\rightarrow +\infty}\mathrm{dist}(S(t)B,\mathcal{A})=0,
\end{equation*}
where~$\mathrm{dist}(A,B):=\sup_{x\in A}\mathrm{dist}_{X}(x,B)$~is the Hausdorff semi-distance.
\end{enumerate}
\end{definition}

Ma, Wang and Zhong put forward the concept of~$\omega $-limit compact in \cite{MWZH} and proved that~$\omega $-limit compactness is a necessary and sufficient condition for a dissipative dynamical system to possess a global attractor.

\begin{definition}\label{def21-2-10-8}
The dynamical system~$(X, \{S(t)\}_{t\geq0})$~is called to be~$\omega$-limit compact iff for every positively invariant bounded set~$B\subseteq X$~we have~$\alpha\big(S(t)B\big)\rightarrow 0$~as~$t\rightarrow\infty$, where~$\alpha(\cdot)$~is the Kuratowski measure of noncompactness.
\end{definition}

\begin{theorem}\label{thm21-2-10-9}
The dynamical system~$(X, \{S(t)\}_{t\geq0})$~has a global attractor in~$X$~if and only if it is both dissipative and ~$\omega$-limit compact.
\end{theorem}

In the definition of global attractor, the description of its attractiveness to bounded sets is qualitative.  In fact, it may attract bounded sets arbitrarily slowly.  And there is no quantitative information about the attractive rate.  In order to  describe the attractive velocity of compact semi-invariant attracting set quantitatively,  we introduce the following concept of compact $\varphi$-attracting set.

\begin{definition}\label{def21-1-31-1}
Assume that~$\varphi:\mathds{R}^{+}\rightarrow\mathds{R}^{+}$~satisfies~$\varphi(t)\rightarrow 0$~as~$t\rightarrow +\infty$. We call a compact~$\mathcal{A}^*\subseteq X$~a compact $\varphi$-attracting set
 for the dynamical system~$(X, \{S(t)\}_{t\geq0})$, iff~$\mathcal{A}^*$ is positively invariant with respect to~$S(t)$ (namely $S(t)\mathcal{A}^*\subseteq \mathcal{A}^*$ for all~$t\geq0$) and there exists~$t_{0}\in \mathds{R}^+$~such that for every bounded set~$B\subseteq X$~there exists~$t_{B}\geq0$~such that
\begin{equation*}
\begin{split}
\mathrm{ dist}\left(S(t)B, \mathcal{A}^*\right)\leq \varphi(t+t_0-t_{B}),\ \forall t\geq t_{B}.
  \end{split}
\end{equation*}
In particular,
\begin{itemize}
  \item [(i)]~if~$\varphi(t)=Ce^{-\beta t}$~for certain positive constants~$C,\beta$, then~$\mathcal{A}^*$ is called a compact exponential attracting set;
  \item [(ii)]~if~$\varphi(t)=Ct^{-\beta}$~for certain positive constants~$C,\beta$, then~$\mathcal{A}^*$ is called a compact polynomial attracting set;
  \item [(iii)]~if~$\varphi(t)=C(\ln t)^{-\beta }$~~for certain positive constants~$C,\beta$, then~$\mathcal{A}^*$ is called a compact logarithm-polynomial attracting set.
\end{itemize}
\end{definition}

%\begin{definition}\label{def21-1-31-1}
%Assume that~$\varphi:\mathds{R}^{+}\rightarrow\mathds{R}^{+}$~satisfies~$\varphi(t)\rightarrow 0$~as~$t\rightarrow +\infty$. We call a compact~$\mathcal{A}^*\subseteq X$~a~$\varphi$-attractor for the dynamical system~$(X, \{S(t)\}_{t\geq0})$, iff~$\mathcal{A}^*$~is positively invariant with respect to~$S(t)$~and there exists~$t_{0}\in \mathds{R}$~such that for every bounded set~$B\subseteq X$~there exists~$t_{B}\geq0$~such that
%\begin{equation*}
%\begin{split}
%\mathrm{ dist}\left(S(t)B, \mathcal{A}^*\right)\leq \varphi(t+t_0-t_{B}),\ \forall t\geq t_{B}.
%  \end{split}
%\end{equation*}
%In particular, if~$\varphi(t)=Ct^{-\beta}$~for certain positive constants~$C,\beta$, then~$\mathcal{A}^*$~is called a polynomial attractor.
%\end{definition}

\begin{remark}We emphasize that the finiteness of fractal dimension is not required in the above definition of compact $\varphi$-attracting set.
\end{remark}

It is not difficult to verify the following lemma.

\begin{lemma}\label{lem21-8-28-1}
Let~$\{S(t)\}_{t\geq0}$~be a dynamical system on a complete metric space~$(X,d)$. Assume that ~$\{S(t)\}_{t\geq0}$~possesses a positively invariant, compact attracting set $K$, i.e., $K\subseteq X$ is compact, positively invariant with respect to~$S(t)$, and  for every bounded set~$B\subseteq X$~we have
  \begin{equation*}
\lim_{t\rightarrow +\infty}\mathrm{dist}(S(t)B,K)=0.
\end{equation*}
 Then~$\mathcal{A}=\omega(K)\subseteq K$ is its global attractor.
\end{lemma}

Lemma~$\ref{lem21-8-28-1}$ indicates that a dynamical system with a compact $\varphi$-attracting set$ ~\mathcal{A}^*$ possesses the global attractor~$\mathcal{A}=\omega(\mathcal{A}^*)\subseteq \mathcal{A}^*$.

We introduce the following concept of~$\varphi$-decay with respect to noncompactness measure, which will be proved in Theorem ~$\ref{20-8-5-3}$ below as a sufficient condition for the existence of compact $\varphi$-attracting sets.
\begin{definition}\label{def20-8-5-3}
The dynamical system~$(X, \{S(t)\}_{t\geq0})$~is called to be~$\varphi$-decaying with respect to noncompactness measure iff it is dissipative and there exists~$t_{0}>0$~such that
  \begin{equation}\label{20-8-5-1}
  \alpha(S(t)\mathcal{B}_0)\leq\varphi(t),\ \forall t\geq t_{0},
\end{equation}
where~$\mathcal{B}_0$~is a positively invariant bounded absorbing set of~$(X, \{S(t)\}_{t\geq0})$~and~$\varphi:\mathds{R}^{+}\rightarrow\mathds{R}^{+}$~is a decreasing function satisfying~$\varphi(t)\rightarrow 0$~as~$t\rightarrow +\infty$.

In particular,
\begin{itemize}
  \item [(i)]~if~$\varphi(t)=Ce^{-\beta t}$~for certain positive constants~$C,\beta$, then~$(X, \{S(t)\}_{t\geq0})$~is said to be exponentially decaying with respect to noncompactness measure;
  \item [(ii)]~if~$\varphi(t)=Ct^{-\beta}$~~for certain positive constants~$C,\beta$, then~$(X, \{S(t)\}_{t\geq0})$~is said to be polynomially decaying with respect to noncompactness measure;
  \item [(iii)]~if~$\varphi(t)=C(\ln t)^{-\beta }$~~for certain positive constants~$C,\beta$, then~$(X, \{S(t)\}_{t\geq0})$~is said to be logarithm-polynomially decaying with respect to noncompactness measure.
\end{itemize}
\end{definition}

The following  lemma shows that the decay rate of noncompact measure of a positively invariant bounded absorbing set determines the decay rate of noncompact measure of any bounded set, which guarantees  the justification of Definition~$\ref{def20-8-5-3}$.

\begin{lemma}\label{lemma 8-20-5-12}
 Assume that the dynamical system~$(X, \{S(t)\}_{t\geq0})$~is~$\varphi$-decaying with respect to noncompactness measure, which implies that there exist a positively invariant bounded absorbing set~$\mathcal{B}_0$~and a positive constant~$t_{0}$~such that
\begin{equation*}
  \alpha(S(t)\mathcal{B}_0)\leq\varphi(t),\ \forall t\geq t_{0}.
\end{equation*}
Then for every bounded subset~$B$~of~$X$, we have
\begin{equation*}
\begin{split}
  \alpha(S(t)B)\leq\varphi(t-t_{*}(B)),\ \forall t\geq t_{*}(B)+t_{0},
  \end{split}
\end{equation*}
where~$t_{*}(B)$~is the entering time of~$B$~into~$\mathcal{B}_0$.
\end{lemma}

\begin{proof}It follows from
\begin{equation*}
  S(t)B\subseteq \mathcal{B}_0, \ \forall t\geq t_{*}(B)
\end{equation*}
that~$S(t)B=S(t-t_{*}(B))S(t_{*}(B))B\subseteq S(t-t_{*}(B))\mathcal{B}_0$.
Consequently by Lemma~$\ref{lemma21-2-10-1}$~we have
\begin{equation*}
\begin{split}
\alpha(S(t)B)\leq\alpha(S(t-t_{*}(B))\mathcal{B}_0)\leq\varphi(t-t_{*}(B)),\ \forall t\geq t_{*}(B)+t_0.
  \end{split}
\end{equation*}
\end{proof}

\section{The existence of compact $\varphi$-attracting sets}
In this section, we will  present a key theorem which illustrates that we can obtain the specific estimate of the attractive velocity of a compact positively invariant attracting set  by estimating the decay rate of the noncompactness measure.
\begin{theorem}\label{20-8-5-3}
Assume that the dynamical system~$(X, \{S(t)\}_{t\geq0})$~is~$\varphi$-decaying with respect to noncompactness measure, which implies that there exist a positively invariant bounded absorbing set~$\mathcal{B}_0$~and a positive constant~$t_{0}$~such that
\begin{equation*}
  \alpha(S(t)\mathcal{B}_0)\leq\varphi(t),\ \forall t\geq t_{0}.
\end{equation*}
Then~$(X, \{S(t)\}_{t\geq0})$~possesses a compact $\varphi$-attracting set~$\mathcal{A}^*$ such that for every bounded set~$B\subseteq X$ we have
\begin{equation}\label{20-8-5-33}
\begin{split}
\mathrm{ dist}\left(S(t)B, \mathcal{A}^*\right)\leq \varphi(t-t_{*}(B)-1),\ \forall t\geq t_{*}(B)+t_{0}+1,
  \end{split}
\end{equation}
where~$t_{*}(B)$~is the entering time of~$B$~into~$\mathcal{B}_0$.
\end{theorem}

\begin{proof}
\textbf{Step 1}:
It follows from
\begin{equation*}
\left\{
\begin{array}{rl}
&\varphi(t)\rightarrow 0\  \text{as}\ t\rightarrow +\infty,\\
&\alpha(S(t)\mathcal{B}_0)\leq\varphi(t),\ \forall t\geq t_{0}
\end{array}
\right.
\end{equation*}
that
\begin{equation}\label{20-8-5-21}
\alpha(S(t)\mathcal{B}_0)\rightarrow 0\  (t\rightarrow +\infty),
\end{equation}
from which we can deduce that~$\mathcal{A}=\omega (\mathcal{B}_0)\equiv \bigcap_{t\geq 0}\overline{S(t)\mathcal{B}_0}$ is the global attractor of this system.

Since
\begin{equation*}
  \alpha(S(t)\mathcal{B}_0)\leq\varphi(t),\ \forall t\geq t_{0},
\end{equation*}
for every positive integer~$m\geq t_{0}$, there exists \[E_m=\bigcup_{i=1}^{K_m}S(m)a_i^{(m)},\  a_i^{(m)}\in \mathcal{B}_0,\] such that $S(m)\mathcal{B}_0\subseteq \bigcup_{x_{\lambda}\in E_m}B(x_{\lambda},\varphi(m))$.

Let~$F_m=\bigcup_{t\geq0}S(t)E_m=\bigcup_{t\geq0}S(t)\bigcup_{i=1}^{K_m}S(m)a_i^{(m)},\  a_i^{(m)}\in \mathcal{B}_0$. Take~$\mathcal{A}^*=\bigcup_{m\in \mathds{N},~m\geq t_{0}}F_m\bigcup\mathcal{A}$. It is obvious that $\mathcal{A}^*$ is positively invariant.

Since~$\mathcal{A}^*\supseteq\bigcup_{m=[t_0]+1}^{+\infty}E_m$~(where~$[t_0]$~is the integer part
 of~$t_0$),
\begin{equation}\label{20-8-5-2}
  \mathrm{dist}\left(S(m)\mathcal{B}_0, \mathcal{A}^*\right)\leq \varphi(m), \ \forall m\geq t_{0}.
\end{equation}
By the positive invariance of~$\mathcal{B}_0$,~$S(t)\mathcal{B}_0\subseteq S([t])\mathcal{B}_0$. Hence we deduce from the monotonicity of $\varphi(t)$ and $(\ref{20-8-5-2})$ that
\begin{equation}\label{20-8-5-4}
\begin{split}
 \mathrm{dist}\left(S(t)\mathcal{B}_0, \mathcal{A}^*\right)\leq &\mathrm{dist}\left(S([t])\mathcal{B}_0, \mathcal{A}^*\right)\\
 \leq &\varphi([t])\\
 \leq &\varphi(t-1)
  \end{split}
\end{equation}
holds for all~$t\geq t_{0}+1$.
For every bounded set~$B\subseteq X$, there exists~$t_{*}(B)$ such that
\begin{equation*}
  S(t)B\subseteq \mathcal{B}_0, \ \forall t\geq t_{*}(B).
\end{equation*}
Therefore~$S(t)B=S(t-t_{*}(B))S(t_{*}(B))B\subseteq S(t-t_{*}(B))\mathcal{B}_0$, and thus for every $t\geq t_{*}(B)+t_{0}+1$ we have
\begin{equation}\label{20-8-5-5}
\begin{split}
 \mathrm{dist}\left(S(t)B, \mathcal{A}^*\right)\leq &\mathrm{dist}\left(S(t-t_{*}(B))\mathcal{B}_0, \mathcal{A}^*\right)\\
 \leq &\varphi(t-t_{*}(B)-1).
  \end{split}
\end{equation}

\textbf{Step 2}:~Next we shall verify that $\mathcal{A}^*$~is compact, i.e., every sequence $\{x_n\}_{n=1}^{+\infty}\subseteq\mathcal{A}^*$ has a subsequence which converges to a point in $\mathcal{A}^*$.
Since the global attractor~$\mathcal{A}$~is compact, if there exists a subsequence $\{x_{n_k}\}_{k=1}^{+\infty}\subseteq\mathcal{A}$, then $\{x_{n_k}\}_{k=1}^{+\infty}$ has a convergent subsequence. Therefore, without loss of generality, we can assume that
~$\{x_n\}_{n=1}^{+\infty}\subseteq\bigcup_{m\in \mathds{N},~m\geq t_{0}}F_m$. Write
\begin{equation}\label{20-8-5-10}
x_n=S(t_n)S(m_n)a_{i_n}^{(m_n)}, \ t_n\geq 0, m_n\geq t_{0}, 1\leq i_n\leq k_{m_n}, a_{i_n}^{(m_n)}\in \mathcal{B}_0.
\end{equation}
\begin{itemize}
  \item[(i)] If~$\{t_n+m_n\}_{n=1}^{+\infty}$ is unbounded, then there exists a subsequence~$\{n_k\}$~of~$\{n\}$ such that $t_{n_k}+m_{n_k}\rightarrow+\infty$ as $k\rightarrow+\infty$. Hence we can deduce from $(\ref{20-8-5-21})$~that there exists a subsequence of $\{x_n\}_{n=1}^{+\infty}$~ convergent in~$\mathcal{A}=\omega(\mathcal{B}_0)\subseteq \mathcal{A}^*$.
  \item[(ii)] If there exists a positive integer~$N_0$~such that~$t_n+m_n\leq N_0$ for all $n\in \mathds{N}$, then~$\{x_n\}\subseteq \bigcup_{m\in \mathds{N},~t_{0}\leq m\leq N_0 }\bigcup_{i=1}^{K_m}\bigcup_{t\in [0,N_0]}S(t)S(m)a_i^{(m)}$. For given~$m$~and~$i$, mapping $t\rightarrow S(t)S(m)a_i^{(m)}$ is continuous, and $[0,N_0]$ is a compact set, so $\bigcup_{t\in [0,N_0]}S(t)S(m)a_i^{(m)}$ is compact, and thus
      \[\bigcup_{m\in \mathds{N},\ t_{0}\leq m\leq N_0 }\bigcup_{i=1}^{K_m}\bigcup_{t\in [0,N_0]}S(t)S(m)a_i^{(m)}\] is also compact.
      Consequently, there exists a subsequence of $\{x_n\}_{n=1}^{+\infty}$ convergent in \[\bigcup_{m\in \mathds{N}, t_{0}\leq m\leq N_0 }\bigcup_{i=1}^{K_m}\bigcup_{t\in [0,N_0]}S(t)S(m)a_i^{(m)}\subseteq\mathcal{A}^*.\]
\end{itemize}

This finishes the proof.
\end{proof}

\begin{remark}\label{remark20-9-14-1}
The idea of constructing a compact $\varphi$-attracting set is to add a countable collection of points to the global attractor such that the added points  can attract a positively  invariant bounded absorbing set~$\mathcal{B}_0$ at the so-called~$\varphi$-speed and thus can  attract any bounded set~$B$ at the~$\varphi$-speed. The decay rate of the noncompactness measure $\alpha$ of a positively invariant bounded absorbing  set $\mathcal{B}_0$
\begin{equation*}
  \alpha(S(t)\mathcal{B}_0)\leq\varphi(t),\ \forall t\geq t_{0} \ (\text{see~$(\ref{20-8-5-1})$})
\end{equation*}
can guarantee that the union of the global attractor and these added points
 is compact. Therefore,  the decay rate with respect to noncompact measure $\alpha$ essentially  characterizes the attractive rate that a positively invariant compact set may reach.
\end{remark}

Due to the non-uniqueness of finite covers of $S(m)\mathcal{B}_0$ by open balls of radius less than $\varphi(m)$, compact $\varphi$-attracting set~$\mathcal{A}^*$~
constructed in  Theorem~$\ref{20-8-5-3}$ is not unique. The following corollary illustrates that for any $\epsilon>0$,  a dynamical system which is~$\varphi$-decaying with respect to noncompactness measure  may have a more regular compact $(\epsilon+1)\varphi$-attracting set.
\begin{corollary}\label{cor20-8-10-1}
Assume that the dynamical system~$(X, \{S(t)\}_{t\geq0})$~is~$\varphi$-decaying with respect to noncompactness measure, which implies that there exist a positively invariant bounded absorbing set~$\mathcal{B}_0$~and a positive constant~$t_{0}$~such that
\begin{equation*}
  \alpha(S(t)\mathcal{B}_0)\leq\varphi(t),\ \forall t\geq t_{0}.
\end{equation*}
Let $D$ be a dense subset of $\mathcal{B}_0$. Then for any $\epsilon>0$, there exists a positively invariant compact set $\mathcal{A}^*$ satisfying:
\begin{itemize}
  \item [(i)]~for every bounded set~$B\subseteq X$~we have
\begin{equation}\label{20-8-10-33}
\begin{split}
 \mathrm{dist}\left(S(t)B, \mathcal{A}^*\right)\leq (\epsilon+1)\varphi(t-t_{*}(B)-1),\ \forall t\geq t_{*}(B)+t_{0}+1,
  \end{split}
\end{equation}
where~$t_{*}(B)$~satisfies
\begin{equation*}
  S(t)B\subseteq \mathcal{B}_0, \ \forall t\geq t_{*}(B);
\end{equation*}
  \item  [(ii)]~$\mathcal{A}^*=\mathcal{A}\cup \mathcal{A}_1$, where $\mathcal{A}$ is the global attractor of $(X, \{S(t)\}_{t\geq0})$, $\mathcal{A}_1$~is a set of orbits starting from  points in $D$.
\end{itemize}
\end{corollary}
\begin{proof}Since
\begin{equation*}
  \alpha(S(t)\mathcal{B}_0)\leq\varphi(t),\ \forall t\geq t_{0},
\end{equation*}
for every positive integer~$m\geq t_{0}$, there exists
\begin{equation*}
E_m=\cup_{i=1}^{K_m}S(m)a_i^{(m)},\  a_i^{(m)}\in \mathcal{B}_0,
\end{equation*}
such that
\begin{equation}\label{20-8-10-34}
  S(m)\mathcal{B}_0\subseteq\cup_{i=1}^{K_m}B(S(m)a_i^{(m)},\varphi(m)).
\end{equation}
Because~$S(t):X\rightarrow X$ is continuous and $D$ is dense in~$\mathcal{B}_0$, there exists $b_i^{(m)}\in D$ such that
\begin{equation}\label{20-8-10-35}
d(S(m)b_i^{(m)},S(m)a_i^{(m)})<\epsilon \varphi(m).
\end{equation}
Let\begin{equation*}
E_m'=\cup_{i=1}^{K_m}S(m)b_i^{(m)}.
\end{equation*}
It follows from~$(\ref{20-8-10-34})$~and~$(\ref{20-8-10-35})$~that
\begin{equation}\label{20-8-10-36}
  S(m)\mathcal{B}_0\subseteq\cup_{i=1}^{K_m}B(S(m)b_i^{(m)},(1+\epsilon)\varphi(m))\equiv\mathcal{N}_{(1+\epsilon)\varphi(m)}E_m' , \forall m\in \mathds{N}, \ m\geq t_0.
\end{equation}
Take~$\mathcal{A}_1=\cup_{m\in \mathds{N},~m\geq t_{0}}\cup_{t\geq0}S(t)E_m'$, $\mathcal{A}^*=\mathcal{A}_1\cup\mathcal{A}$. Obviously, $\mathcal{A}_1$ is a set of orbits starting from  points in $D$ and $\mathcal{A}^*$ is positively invariant. Analysis similar to that in the second step in the proof of Theorem~$\ref{20-8-5-3}$~shows that $\mathcal{A}^*$ is compact. From~$(\ref{20-8-10-36})$~we have
\begin{equation}\label{20-8-10-44}
  S(m)\mathcal{B}_0\subseteq\mathcal{N}_{(1+\epsilon)\varphi(m)}\mathcal{A}^* , \forall m\in \mathds{N}, \ m\geq t_0.
\end{equation}

For every~$B\subseteq X$, there exists~$t_{*}(B)$ such that
\begin{equation*}
  S(t)B\subseteq \mathcal{B}_0, \ \forall t\geq t_{*}(B).
\end{equation*}
Hence
\begin{equation}\label{20-8-10-46}
  S(t)B=S(t-t_{*}(B))S(t_{*}(B))B\subseteq S(t-t_{*}(B))\mathcal{B}_0\subseteq S([t-t_{*}(B)])\mathcal{B}_0.
\end{equation}
We deduce from~$(\ref{20-8-10-44})$~,~$(\ref{20-8-10-46})$~and the  monotonicity of $\varphi$~that
\begin{equation*}
\begin{split}
  \mathrm{dist}(S(t)B,\mathcal{A}^*)\leq&(1+\epsilon)\varphi([t-t_{*}(B)])\\
  \leq &(1+\epsilon)\varphi(t-t_{*}(B)-1)
\end{split}\end{equation*}
holds for all $t\geq t_0+t_{*}(B)+1$.

The proof is completed.
\end{proof}

\begin{remark}
The significance of  Corollary~$\ref{cor20-8-10-1}$  lies in that $\mathcal{A}^{*}$ constructed here  often has higher regularity since the global attractor $\mathcal{A}$ is often more regular than the phase space and the phase space often has dense subspaces with higher regularity.
\end{remark}

\section{Criteria for $\varphi$-decay with respect to noncompactness measure}

In this section, we will give several criteria for $\varphi$-decay with respect to noncompactness measure.
\begin{theorem}\label{20-8-8-1}
Let~$(X, \{S(t)\}_{t\geq0})$~be a dissipative dynamical system and $\mathcal{B}_0$ be its positively invariant bounded absorbing set. If there exist a compact set $A\subseteq X$ and a positive constant $t_{0}$ such that
\begin{equation}\label{20-8-8-2}
\begin{split}
 \mathrm{dist}\left(S(t)\mathcal{B}_0, A\right)\leq \varphi(t),\ \forall t\geq t_{0},
  \end{split}
\end{equation}
where function~$\varphi:\mathds{R}^{+}\rightarrow\mathds{R}^{+}$~satisfies~$\varphi(t)\rightarrow 0$ as $t\rightarrow +\infty$, then ~$\{S(t)\}_{t\geq0}$~is $2\varphi$-decaying with respect to noncompactness measure.
\end{theorem}
\begin{proof}
Let $\epsilon$ be an arbitrary positive number.

Since~$A$ is compact, $A$ has a finite $\epsilon$-net ${x_1,x_2,\ldots,x_n}$, i.e., $A\subseteq \cup_{i=1}^{n}B(x_i,\epsilon)$. By ~$(\ref{20-8-8-2})$, for every~$x\in \mathcal{B}_0$, there exists $y\in A$ such that
\begin{equation*}
  d(S(t)x,y)<\varphi(t)+\epsilon,\ \forall t\geq t_{0}.
\end{equation*}
Therefore there exists~$i\in{1,2,\ldots,n}$~such that
\begin{equation*}
   d(S(t)x,x_i)< d(S(t)x,y)+ d(y,x_i)\leq\varphi(t)+\epsilon+\epsilon,
\end{equation*}
i.e.,~$S(t)\mathcal{B}_0\subseteq \cup_{i=1}^{n}B(x_i,\varphi(t)+2\epsilon)$.
Consequently, by the definition of  Kuratowski's noncompactness measure, we obtain~$\alpha\big(S(t)\mathcal{B}_0\big)\leq2\varphi(t)\ (\forall t\geq t_{0})$ and thus complete the proof.
\end{proof}

\begin{theorem}Let~$\{S(t)\}_{t\geq0}$~be a dissipative semigroup in a~Banach~space~$(X,\|\cdot\|)$ and $\mathcal{B}_0$~be its positively invariant bounded absorbing set. Assume that function~$\varphi:\mathds{R}^{+}\rightarrow\mathds{R}^{+}$~satisfies~$\varphi(t)\rightarrow 0$ as $t\rightarrow +\infty$. If there exists a positive constant $t_0$~such that for any~$\epsilon>0$, there exists a finite-dimensional subspace~$X_1\subseteq X$, such that
\begin{equation}\label{20-8-9-1}
\sup_{x\in \mathcal{B}_0}\|(I-P)S(t)x\|\leq\varphi(t)+\epsilon,\ \forall t\geq t_0,
\end{equation}
where~$P:X\rightarrow X_1$~is a bounded projection operator,
then $\{S(t)\}_{t\geq0}$~is $2\varphi$-decaying with respect to noncompactness measure.
\end{theorem}
\begin{proof}
As a bounded set of
 finite-dimensional subspace  $X_1$,  $PS(t)\mathcal{B}_0$ is precompact, and thus
\begin{equation}\label{20-8-10-22}
 \alpha\big(PS(t)\mathcal{B}_0\big)=0.
\end{equation}
By the definition of Kuratowski's  noncompactness measure, $(\ref{20-8-9-1})$~implies that
\begin{equation}\label{20-8-10-23}
 \alpha\big((I-P)S(t)\mathcal{B}_0\big)\leq 2\varphi(t)+2\epsilon,\ \forall t\geq t_0.
\end{equation}
Since~$S(t)\mathcal{B}_0\subseteq PS(t)\mathcal{B}_0+(I-P)S(t)\mathcal{B}_0$,  combining~$(\ref{20-8-10-22})$~and~$(\ref{20-8-10-23})$~yields that
\begin{equation*}
  \alpha\big(S(t)\mathcal{B}_0\big)\leq\alpha\big(PS(t)\mathcal{B}_0\big)+\alpha\big((I-P)S(t)\mathcal{B}_0\big)\leq 2\varphi(t)+2\epsilon,\ \forall t\geq t_0.
\end{equation*}
By the arbitrariness of~$\epsilon$, we have
\begin{equation*}
  \alpha\big(S(t)\mathcal{B}_0\big)\leq 2\varphi(t),\ \forall t \geq t_0,
\end{equation*}
which completes the proof.
\end{proof}
The next criterion  is based on the method of contractive function  and only involves some
rather weak compactness associated with the repeated limit inferior, therefore it can be used to  prove $\varphi$-decay with respect to noncompact measure for infinite-dimensional dynamical systems under critical conditions.
\begin{definition}
Let $X$ be a complete metric space and $B$ be a bounded subset of $X$. Function ~$\Phi:X\times X\rightarrow \mathds{R^{+}}$~ is said to be contractive on $B\times B$ if it satisfies the following conditions:
\begin{equation*}
 \liminf_{m\rightarrow\infty} \liminf_{n\rightarrow\infty} \Phi(y_n,y_m)=0, \forall \{y_n\}\subseteq B.
\end{equation*}
We denote by~$\mathrm{Contr}(B)$ the set of all contractive functions on $B\times B$.
\end{definition}

\begin{theorem} \label{20-7-24-35}
Let~$\{S(t)\}_{t\geq0}$~be a semigroup on a complete metric space $(X,d)$ and $B\subseteq X$ be its positively invariant bounded set. Assume that function~$\varphi:\mathds{R}^{+}\rightarrow\mathds{R}^{+}$~satisfies~$\varphi(t)\rightarrow 0$ as $t\rightarrow +\infty$ and $T$~is a positive constant. If for every~$t\geq T$~and every~$\epsilon>0$, there exists $\Phi_{t,\epsilon}\in \mathrm{Contr}(B)$~such that
\begin{equation}\label{20-7-24-36}
d(S(t)y_1,S(t)y_2)\leq \varphi(t)+\epsilon+\Phi_{t,\epsilon}(y_1,y_2), \forall y_1,y_2\in B,
\end{equation}
then
\begin{equation}\label{20-7-24-37}
\begin{split}
\alpha(S(t)B)\leq 3\varphi(t),\ \forall t\geq T.
 \end{split}\end{equation}
\end{theorem}
\begin{proof}
If~$(\ref{20-7-24-37})$ was false, then there would exist~$t_0 \geq T$ such that
\begin{equation*}
\begin{split}
\alpha(S(t_0)B)> 3\varphi(t_0).
 \end{split}\end{equation*}
 Hence according to the definition of noncompactness measure,  $S(t_0)B$  has no finite cover of diameter less than or equal to $3\varphi(t_0)$. Consequently, for any $y_1\in B$, there exists~$y_2\in B$~such that
\begin{equation*}
  d(S(t_0)y_1,S(t_0)y_2)>\frac{3}{2}\varphi(t_0).
\end{equation*}
And then we can find ~$y_3\in B$~ satisfying
\begin{equation*}
  d(S(t_0)y_3,S(t_0)y_i)>\frac{3}{2}\varphi(t_0),\ i=1,2.
\end{equation*}
 Iterating in this way, we obtain a sequence $\{y_n\}\subseteq B$ satisfying
\begin{equation}\label{20-7-24-40}
  d(S(t_0)y_n,S(t_0)y_m)>\frac{3}{2}\varphi(t_0),\  \forall m\neq n.
\end{equation}
By~$(\ref{20-7-24-36})$, for~$t_0$~and~$\epsilon=\frac{1}{4}\varphi(t_0)$, there exists $\Phi_{t_0}\in Contr(B)$~such that
\begin{equation}\label{20-7-24-41}
d(S(t_0)y_n,S(t_0)y_m)\leq \varphi(t_0)+\frac{1}{4}\varphi(t_0)+\Phi_{t_0}(y_n,y_m).
\end{equation}
Combining~$(\ref{20-7-24-40})$~and~$(\ref{20-7-24-41})$~yields
\begin{equation}\label{20-7-24-42}
\Phi_{t_0}(y_n,y_m)>\frac{1}{4}\varphi(t_0)>0,
\end{equation}
which implies
\begin{equation}\label{20-7-24-43}
\liminf_{m\rightarrow\infty} \liminf_{n\rightarrow\infty}\Phi_{t_0}(y_n,y_m)\geq\frac{1}{4}\varphi(t_0)>0.
\end{equation}
This is contrary to~$\Phi_{t,\epsilon}\in \mathrm{Contr}(B)$. We have thus proved the theorem.
\end{proof}

The following theorem states that quasi-stable systems are exponentially decaying with respect to noncompact measure.
\begin{theorem}\label{20-7-27-80}
Let~$\{S(t)\}_{t\geq0}$~be a dissipative semigroup on a complete metric space $(X,d)$ and $\mathcal{B}_0$ be its positively invariant bounded absorbing set. Assume that there exist positive constants $T, \delta_0$, a constant~$\eta\in[0,1)$, a function $g:(\mathds{R}^{+})^{m}\rightarrow\mathds{R}^{+}$~and pseudometrics ~$\varrho_{i}\ (i=1,2,\ldots,m)$~on~$\mathcal{B}_0$~such that
\begin{itemize}
\item[(i)]~$g$~is non-decreasing with respect to each
variable,~$g(0,\ldots,0)=0$~and~$g$~is continuous at~$(0,\ldots,0)$;
\item[(ii)] for each $i = 1, 2, . . . , m$,~$\varrho_{i}$~is precompact on~$\mathcal{B}_0$, i.e., any sequence~$\{x_n\}\subseteq \mathcal{B}_0$~has a subsequence~$\{x_{n_k}\}$~which is Cauchy with
respect to~$\varrho_{i}$;
\item[(iii)]~the inequality
\begin{equation}\label{20-8-10-9}
\begin{split}
  &d(S(T)y_1,S(T)y_2)\\\leq &\eta d(y_1,y_2)+g\Big(\varrho_1\big(y_1,y_2\big),\varrho_2\big(y_1,y_2\big),\ldots,\varrho_m\big(y_1,y_2\big)\Big)
\end{split}\end{equation}
 holds for all~$y_1,y_2\in \mathcal{B}_0$~satisfying
~$\varrho_{i}(y_1,y_2)\leq \delta_0\ (i=1,2,\ldots,m)$.
\end{itemize}
Then~$(X, \{S(t)\}_{t\geq0})$~is exponentially decaying with respect to noncompactness measure. And for each bounded~$B\subseteq X$~we have
\begin{equation*}
\begin{split}
  \alpha(S(t)B)\leq \eta^{\frac{t-t_{*}(B)-T}{T}}\alpha\big(\mathcal{B}_0\big),\ \forall t\geq t_{*}(B)+T,
  \end{split}
\end{equation*}
where~$t_{*}(B)$~satisfies
\begin{equation*}
  S(t)B\subseteq \mathcal{B}_0, \ \forall t\geq t_{*}(B).
\end{equation*}
\end{theorem}
\begin{proof}
Let~$D\subseteq \mathcal{B}_0$.

For any~$\epsilon>0$, there exist sets~$F_1,F_2,\ldots,F_n$ such that
\begin{equation}\label{20-8-10-1}
 D\subseteq\cup_{j=1}^{n}F_n,\ \mathrm{diam} F_j<\alpha(D)+\epsilon.
\end{equation}
It follows from assumption (i) that there exists
$\delta>0$~such that~$g(x_1,x_2,\ldots,x_m)<\epsilon$ whenever~$x_i\in [0,\delta]\ (i=1,2,\ldots,m)$.
By the precompactness of~$\varrho_i\ (i=1,2,\ldots,m)$, there exists a finite set ~$\mathcal{N}^i=\{x^i_{j}:j=1,2,\ldots,k_i\}\subseteq D$~such that for every~$y\in D$~there is~$x^i_{j}\in \mathcal{N}^i$~with the property ~$\textcolor[rgb]{1.00,0.00,0.00}{\varrho_i}(y,x^i_{j})\leq\frac{1}{2}\min\{\delta,\delta_0\}$, i.e.,
\begin{equation}\label{20-8-10-2}
 D\subseteq\cup_{j=1}^{k_i}C_{j}^{i},\ C_{j}^{i}=\big\{y:\varrho_i(y,x^i_{j})\leq\frac{1}{2}\min\{\delta,\delta_0\}\big\},\ i=1,\ldots,m.
\end{equation}
Consequently, we have
\begin{equation*}
  D\subseteq\cup_{j_1=1}^{k_1}\cup_{j_2=1}^{k_2}\cdots\cup_{j_m=1}^{k_m}\cup_{j=1}^{n}( C_{j_1}^{1}\cap C_{j_2}^{2}\cap\ldots\cap C_{j_m}^{m}\cap F_j)
\end{equation*}
and
\begin{equation*}
  S(T)D\subseteq\cup_{j_1=1}^{k_1}\cup_{j_2=1}^{k_2}\cdots\cup_{j_m=1}^{k_m}\cup_{j=1}^{n}\big(S(T)( C_{j_1}^{1}\cap C_{j_2}^{2}\cap\ldots\cap C_{j_m}^{m}\cap F_j)\big).
\end{equation*}
By~$(\ref{20-8-10-1})$~and~$(\ref{20-8-10-2})$, for any~$y_1,y_2\in C_{j_1}^{1}\cap C_{j_2}^{2}\cap\ldots\cap C_{j_m}^{m}\cap F_j$, we have
\begin{equation}\label{20-8-10-3}
  d\big(y_1,y_2\big)\leq \mathrm{diam} F_j<\alpha(D)+\epsilon
\end{equation}
and
\begin{equation}\label{20-8-10-4}
  \varrho_i\big(y_1,y_2\big)\leq \min\{\delta,\delta_0\}, i=1,2,\ldots,m.
\end{equation}
Inequality~$(\ref{20-8-10-4})$~implies
\begin{equation}\label{20-8-10-5}
  g\Big(\varrho_1\big(y_1,y_2\big),\ldots,\varrho_m\big(y_1,y_2\big)\Big)<\epsilon.
\end{equation}
We deduce from~$(\ref{20-8-10-9})$,~$(\ref{20-8-10-3})$-$(\ref{20-8-10-5})$~that
\begin{equation*}
\begin{split}
  d(S(T)y_1,S(T)y_2)\leq \eta \big(\alpha(D)+\epsilon\big)+\epsilon,
\end{split}\end{equation*}
i.e.,
\begin{equation*}
  \mathrm{diam}\big(S(T)( C_{j_1}^{1}\cap C_{j_2}^{2}\cap\ldots\cap C_{j_m}^{m}\cap F_j)\big)\leq \eta \big(\alpha(D)+\epsilon\big)+\epsilon.
\end{equation*}
Hence according to the definition of Kuratowski's~noncompactness measure,
\begin{equation}\label{20-8-10-11}
  \alpha(S(T)D)\leq \eta  \alpha(D).
\end{equation}
From~$(\ref{20-8-10-11})$~we obtain
\begin{equation*}
  \alpha\big(S(nT)\mathcal{B}_0\big)=\alpha\big(S(T)S((n-1)T)\mathcal{B}_0\big)\leq\eta\alpha\big(S((n-1)T)\mathcal{B}_0\big), \forall n\in \mathds{N},
\end{equation*}
which by iteration yields
\begin{equation*}
  \alpha\big(S(nT)\mathcal{B}_0\big)\leq\eta^n\alpha\big(\mathcal{B}_0\big).
\end{equation*}
Therefore, for all~$t\geq T$~we have
\begin{equation}\label{20-8-10-15}
\begin{split}
  \alpha\big(S(t)\mathcal{B}_0\big)\leq &\alpha\big(S([\frac{t}{T}]T)\mathcal{B}_0\big)\\ \leq&\eta^{[\frac{t}{T}]}\alpha\big(\mathcal{B}_0\big)\\ \leq&\eta^{\frac{t}{T}-1}\alpha\big(\mathcal{B}_0\big).
\end{split}\end{equation}
Thus by Lemma~$\ref{lemma 8-20-5-12}$, for every bounded set~$B\subseteq X$, we have an estimate
\begin{equation*}
\begin{split}
  \alpha(S(t)B)\leq\eta^{\frac{t-t_{*}(B)-T}{T}}\alpha\big(\mathcal{B}_0\big),\ \forall t\geq t_{*}(B)+T,
  \end{split}
\end{equation*}
where~$t_{*}(B)$~satisfies
\begin{equation*}
  S(t)B\subseteq \mathcal{B}_0, \ \forall t\geq t_{*}(B).
\end{equation*}
\end{proof}

\begin{remark}
When~$g=0$,~$(\ref{20-8-10-9})$ implies that~$\{S(t)\}_{t\geq0}$ is exponentially stable in~$B\subseteq X$.~$(\ref{20-8-10-9})$ is called quasi-stability inequality since it pertains to the decomposition of the flow into exponentially stable and
compact parts.
\end{remark}

In the following theorem,  a contractive function term is added to the quasi-stable condition, which gives a criterion for the existence of compact exponential attracting sets in the critical case.

\begin{theorem}\label{22-1-28-1}
Let~$\{S(t)\}_{t\geq0}$~be a dissipative semigroup on a complete metric space $(X,d)$ and $\mathcal{B}_0$ be its positively invariant bounded absorbing set. Assume that there exist a positive constant $T$, a constant~$\eta\in[0,1)$, a function $g:(\mathds{R}^{+})^{m}\rightarrow\mathds{R}^{+}$, a function ~$\phi:X\times X\rightarrow \mathds{R^{+}}$ and pseudometrics ~$\varrho_{i}\ (i=1,2,\ldots,m)$~on~$\mathcal{B}_0$~such that
\begin{itemize}
\item[(i)]~$g$~is non-decreasing with respect to each
variable,~$g(0,\ldots,0)=0$~and~$g$~is continuous at~$(0,\ldots,0)$;
\item[(ii)]for any $\{y_n\}\subseteq \mathcal{B}_0$ there exists a subsequence $\{y_{n_k}\}$ of $\{y_n\}$ such that
\begin{equation*}
 \lim_{k\rightarrow\infty} \lim_{l\rightarrow\infty} \phi(y_{n_k},y_{n_l})=0;
\end{equation*}
\item[(iii)] for each $i = 1, 2, . . . , m$,~$\varrho_{i}$~is precompact on~$\mathcal{B}_0$, i.e., any sequence~$\{x_n\}\subseteq \mathcal{B}_0$~has a subsequence~$\{x_{n_k}\}$~which is Cauchy with
respect to~$\varrho_{i}$;
\item[(iv)]~the inequality
\begin{equation}\label{22-1-28-3}
\begin{split}
  d(S(T)y_1,S(T)y_2)\leq \eta d(y_1,y_2)+g\Big(\varrho_1\big(y_1,y_2\big),\varrho_2\big(y_1,y_2\big),\ldots,\varrho_m\big(y_1,y_2\big)\Big)
+\phi\big(y_1,y_2\big)\end{split}\end{equation}
 holds for all~$y_1,y_2\in \mathcal{B}_0$.
\end{itemize}
Then~$(X, \{S(t)\}_{t\geq0})$~is exponentially decaying with respect to noncompactness measure. And for each bounded~$B\subseteq X$~we have
\begin{equation*}
\begin{split}
  \alpha(S(t)B)\leq 2\eta^{\frac{t-t_{*}(B)-T}{T}}\alpha\big(\mathcal{B}_0\big),\ \forall t\geq t_{*}(B)+T,
  \end{split}
\end{equation*}
where~$t_{*}(B)$~satisfies
\begin{equation*}
  S(t)B\subseteq \mathcal{B}_0, \ \forall t\geq t_{*}(B).
\end{equation*}
\end{theorem}
\begin{proof}
Iterating~$(\ref{22-1-28-3})$  we deduce that for any~$y_1,y_2\in \mathcal{B}_0$ and each $n\in\mathds{N}$
\begin{equation}\label{22-1-28-17}
\begin{split}
  d(S(nT)y_1,S(nT)y_2)\leq &\eta d(S((n-1)T)y_1,S((n-1)T)y_2)\\&+g\Big(\varrho_1\big(S((n-1)T)y_1,S((n-1)T)y_2\big),\ldots,\varrho_m\big(S((n-1)T)y_1,S((n-1)T)y_2\big)\Big)\\&
+\phi\big(S((n-1)T)y_1,S((n-1)T)y_2\big)\\
\leq &\eta^n d(y_1,y_2)+\sum\limits_{k=0}^{n-1}\eta^{n-k-1}\bigg[g\Big(\varrho_1\big(S(kT)y_1,S(kT)y_2\big),\ldots,\varrho_m\big(S(kT)y_1,S(kT)y_2\big)\Big)
\\&+\phi\big(S(kT)y_1,S(kT)y_2\big)\bigg].
\end{split}
\end{equation}
By the definition of the noncompactness measure~$\alpha$, for any $\epsilon>0$, there exist sets~$F_1,F_2,\ldots,F_\gamma$~such that
\begin{equation}\label{22-1-28-11}
  \mathcal{B}_0\subseteq\cup_{j=1}^{\gamma}F_j,\ \mathrm{diam} F_j<\alpha(\mathcal{B}_0)+\epsilon.
\end{equation}
For every sequence $\{y_p\}\subseteq \mathcal{B}_0$, there exist~$j_{*}\in\{1,2,\ldots,\gamma\}$ and a subsequence~$\{y_{p_\iota}\}$ of $\{y_p\}$ such that
$\{y_{p_\iota}\}_{\iota=1}^{+\infty}\subseteq F_{j_*}$. By~$(\ref{22-1-28-11})$, we have
\begin{equation}\label{22-1-28-12}
  d\big(y_{p_\iota},y_{p_\nu}\big)\leq \mathrm{diam} F_{j_*}<\alpha(\mathcal{B}_0)+\epsilon, ~~\forall \iota,\nu\in \mathds{N}.
\end{equation}
According to  assumptions (i),~(ii) and (iii) of the theorem,~$\{y_{p_\iota}\}$ has a subsequence ~$\{y_{p_{\iota_{\lambda}}}\}$  such that
\begin{equation}\label{22-1-28-18}
 \lim_{\lambda,\chi\rightarrow+\infty} g\Big(\varrho_1\big(S(kT)y_{p_{\iota_{\lambda}}},S(kT)y_{p_{\iota_{\chi}}}\big),\ldots,\varrho_m\big(S(kT)y_{p_{\iota_{\lambda}}},S(kT)y_{p_{\iota_{\chi}}}\big)\Big)=0
\end{equation}
and
\begin{equation}\label{22-1-28-19}
 \lim_{\lambda\rightarrow+\infty} \lim_{\chi\rightarrow+\infty}  \phi\big(S(kT)y_{p_{\iota_{\lambda}}},S(kT)y_{p_{\iota_{\chi}}}\big)=0
\end{equation}
hold for~$k=0,\ldots,n-1$.

We deduce from~$(\ref{22-1-28-17})$,~$(\ref{22-1-28-12})$-$(\ref{22-1-28-19})$ that
\begin{equation}\label{22-1-28-30}
    \liminf_{p\rightarrow +\infty} \liminf_{q\rightarrow +\infty}  d(S(nT)y_{p},S(nT)y_{q})\leq\liminf_{\lambda\rightarrow+\infty} \liminf_{\chi\rightarrow+\infty}  d(S(nT)y_{p_{\iota_{\lambda}}},S(nT)y_{p_{\iota_{\chi}}})\leq  \eta^n\big(\alpha(\mathcal{B}_0)+\epsilon\big),
\end{equation}
which implies that
\begin{equation}\label{22-1-28-31}
    \liminf_{p\rightarrow +\infty} \liminf_{q\rightarrow +\infty}  d(S(nT)y_{p},S(nT)y_{q})\leq \eta^n\alpha(\mathcal{B}_0),~~\forall\{y_p\}\subseteq \mathcal{B}_0,~\forall n\in\mathds{N}.
\end{equation}
Inequality~$(\ref{22-1-28-31})$ implies that
\begin{equation}\label{22-1-28-35}
\begin{split}
\alpha\big(S(nT)\mathcal{B}_0\big)\leq 2\eta^n\alpha(\mathcal{B}_0).
 \end{split}\end{equation}
 Indeed, if~$(\ref{22-1-28-35})$ was false,  then there would exist~$d_0>0$ such that
\begin{equation}\label{22-1-28-40}
\begin{split}\alpha\big(S(nT)\mathcal{B}_0\big)>d_0>  2\eta^n\alpha(\mathcal{B}_0).
 \end{split}\end{equation}
Hence according to the definition of the noncompactness measure~$\alpha$,~$S(nT)\mathcal{B}_0$  has no finite cover of diameter less than or equal to $d_0$. Consequently, for any $y_1\in \mathcal{B}_0$, there exists~$y_2\in \mathcal{B}_0$ such that
\begin{equation*}
  d(S(nT)y_1,S(nT)y_2)>\frac{d_0}{2}.
\end{equation*}
And then we can find ~$y_3\in \mathcal{B}_0$~ satisfying
\begin{equation*}
  d(S(nT)y_3,S(nT)y_i)>\frac{d_0}{2},\ i=1,2.
\end{equation*}
 Iterating in this way, we obtain a sequence $\{y_p\}\subseteq \mathcal{B}_0$ satisfying
\begin{equation}\label{21-9-1-16}
  d(S(nT)y_p,S(nT)y_q)>\frac{d_0}{2}> \eta^n\alpha(\mathcal{B}_0),\  \forall p\neq q.
\end{equation}
This is contrary to~$(\ref{22-1-28-31})$.

Therefore, for all~$t\geq T$~we have
\begin{equation}\label{20-8-10-15}
\begin{split}
  \alpha\big(S(t)\mathcal{B}_0\big)\leq 2\eta^{\frac{t}{T}-1}\alpha\big(\mathcal{B}_0\big).
\end{split}\end{equation}
Thus by Lemma~$\ref{lemma 8-20-5-12}$, for every bounded set~$B\subseteq X$, we have an estimate
\begin{equation*}
\begin{split}
  \alpha(S(t)B) \leq 2\eta^{\frac{t-t_{*}(B)-T}{T}}\alpha\big(\mathcal{B}_0\big),\ \forall t\geq t_{*}(B)+T,
  \end{split}
\end{equation*}
where~$t_{*}(B)$~satisfies
\begin{equation*}
  S(t)B\subseteq \mathcal{B}_0, \ \forall t\geq t_{*}(B).
\end{equation*}
\end{proof}

\section{Application to a class of wave equations}\label{20-8-10-67}
In the last section, we will establish the existence of compact exponential attracting sets for a class of wave equations with critical nonlinearity by employing  Theorem~$\ref{20-8-5-3}$, Theorem~$\ref{20-7-24-35}$ and Theorem~$\ref{22-1-28-1}$.

Let~$\Omega$ be a bounded domain in $\mathbb{R}^N(N\geq3)$ with a sufficiently smooth boundary~$\partial\Omega$. Below we will denote the inner product and the norm on~$L^2(\Omega)$~by~$(\cdot,\cdot)$ and $\|\cdot\|$ respectively, and the norm on~$L^p(\Omega)$~by~$\|\cdot\|_p$. Consider the wave equation:
\begin{align}
\label{wave equa3}&u_{tt}-\Delta u+k||u_{t}||^p u_t+lu_t+f(u)
=\displaystyle\int_{\Omega}K(x,y)u_{t}(y)dy+h(x), \ x\in\Omega,t\geq0,\\
\label{boundary condition3}&u|_{\partial \Omega}=0, \ t\geq0,\\
\label{initial condition3}&u(x,0)=u_0(x) , u_{t}(x,0)=u_{1}(x),\  x\in\Omega.
\end{align}
We assume the following.
\begin{assumption}\label{22-2-5-1}
\begin{itemize}
  \item [(i)]~$k$,~$l$~and $p$ are positive constants, $K\in L^2(\Omega\times\Omega)$, $h\in L^2(\Omega)$;
  \item [(ii)]~$f\in C^1(\mathbb{R})$ satisfies the critical growth condition
\begin{equation}\label{growth}
    |f'(s)| \leq M(|s|^{\frac{2}{N-2}}+1)
  \end{equation}
and the dissipativity condition
\begin{equation}\label{dissipativity condition}
    \liminf_{|s|\rightarrow\infty}f'(s)\equiv \mu > -\lambda_{1},
  \end{equation}
where $M\geq0$ and $\lambda_{1}$ is the first eigenvalue of the operator $-\Delta$ equipped with Dirichlet boundary condition.
\end{itemize}
\end{assumption}
\begin{definition}
A function $u \in C([0,T];H^1_0(\Omega)) \cap C^1([0,T];L^2(\Omega))$ with $u(0)=u_0$ and $u_t(0)=u_1$ is said to be a weak solution  to problem $(\ref{wave equa3})$-$(\ref{initial condition3})$ on the interval $[0,T]$, iff $u$ satisfies Eq. $(\ref{wave equa3})$ in the sense of distributions.
\end{definition}
 We have proved in \cite{my1} the global well-posedness and dissipativity  for  the problem~$(\ref{wave equa3})$-$(\ref{initial condition3})$ with~$l = 0 $.  Adding  the linear damping term~$lu_t~(l>0)$  makes no essential difference to  the proof of the well-posedness and dissipativity.  Similarly as in \cite{my1}, we have the lemma  as follows.

\begin{lemma}\cite{my1}\label{21-8-29-7}
Let~$T> 0$~be arbitrary. Under Assumption~$\ref{22-2-5-1}$, for every~$(u_0,u_1) \in H^1_0(\Omega)\times L^2(\Omega)$, the initial
boundary value problem~$(\ref{wave equa3})$-$(\ref{initial condition3})$ has a unique weak solution~$u \in C([0,T];H^1_0(\Omega)) \cap C^1([0,T];L^2(\Omega))$, which generates the semigroup
\begin{equation*}
  S(t)(u_0,u_1) = (u(t),u_t (t)),\ t\geq0
\end{equation*}
on~$H^1_0(\Omega)\times L^2(\Omega) $.

Furthermore, the semigroup~$\{S(t)\}_{t\geq0}$~is dissipative, which implies the existence of a positively invariant bounded absorbing set~$\mathcal{B}_0$.
\end{lemma}

To establish the main result in this section, the following lemmas are also needed.

\begin{lemma}\cite{Lax}\label{lemma 10-10-3}
The integral operator
\begin{eqnarray*}
K:L^{2}(\Omega  )&\longrightarrow &L^{2}(\Omega) \\
v&\longmapsto &\int_{\Omega}K(x,y)v(y)dy
\end{eqnarray*}
is a compact operator provided that the kernel $K(x,y)$ is square integrable.
\end{lemma}
\begin{lemma}\citep[Corollary 4]{Simon1986}\label{lemma 1-11-1}
Assume $X\hookrightarrow\hookrightarrow B \hookrightarrow Y$ where $X, B, Y$ are Banach spaces. The following statements hold.
\begin{enumerate}[(i)]
\item  Let $F$ be bounded in $L^{p}(0,T;X)$ where $1\leq p<\infty$, and $\partial F/\partial t=\{\partial f/\partial t:f\in F\}$ be bounded in $L^{1}(0,T;Y)$, where $\partial /\partial t$ is the weak time derivative. Then $F$ is relatively compact in $L^{p}(0,T;B)$.
  \item Let $F$ be bounded in $L^{\infty}(0,T;X)$ and $\partial F/\partial t$ be bounded in $L^{r}(0,T;Y)$ where $r>1$. Then $F$ is relatively compact in $C(0,T;B)$.
\end{enumerate}
\end{lemma}

\begin{lemma}\label{lem 1-14-2}
Let $\{a_{m,n}\}_{m,n=1}^{\infty}$ be a non-negative real sequence with two variables $m,n\in \mathbb{N}$. Then $\liminf\limits_{m\rightarrow \infty}\liminf\limits_{n\rightarrow \infty}a_{m,n}=0$ if and only if for each $\epsilon>0$ and $M\in\mathbb{N}$ there exists $m=m(\epsilon,M)>M$ such that there exists $n=n(m)>N$ satisfying $a_{m,n}\leq\epsilon$ for every $N\in\mathbb{N}$.
\end{lemma}
\begin{proof}
$\liminf\limits_{m\rightarrow \infty}\liminf\limits_{n\rightarrow \infty}a_{m,n}=0$

$\Longleftrightarrow$~for each~$\epsilon>0$~and each~$M\in\mathbb{N}$, there exists~$m=m(\epsilon,M)>M$~such that~$\liminf\limits_{n\rightarrow \infty}a_{m,n}\leq \epsilon$

$\Longleftrightarrow$~for each $\epsilon>0$ and $M\in\mathbb{N}$ there exists $m=m(\epsilon,M)>M$ such that there exists $n=n(m)>N$ satisfying $a_{m,n}\leq\epsilon$ for every $N\in\mathbb{N}$.
\end{proof}

\begin{lemma}\label{corollary 1-14-1}
 $\{a_{m,n}\}_{m,n=1}^{\infty}$ be a non-negative real sequence with two variables $m,n\in \mathbb{N}$. If there exists a subsequence $\{n_k\}_{k=1}^{\infty}$ of $\{n\}_{n=1}^{\infty}$ such that $\lim\limits_{k,l\rightarrow\infty}a_{n_k,n_l}=0$,
then
\begin{equation*}
  \liminf\limits_{m\rightarrow \infty}\liminf\limits_{n\rightarrow \infty}a_{m,n}=0.
\end{equation*}
\end{lemma}
\begin{proof} This lemma can be deduced directly from Lemma~\ref{lem 1-14-2}.
\end{proof}

\begin{theorem}\label{20-8-2-4}
Under Assumption~$\ref{22-2-5-1}$, the dynamical system~$(H_0^1(\Omega)\times L^2(\Omega),\{S(t)\}_{t\geq0})$\ generated by problem~$(\ref{wave equa3})$-$(\ref{initial condition3})$ possesses a compact exponential attracting set~$\mathcal{A}^*$ such that for every bounded set~$B\subseteq X$ we have
\begin{equation}\label{21-8-31-11}
\begin{split}
\mathrm{ dist}\left(S(t)B, \mathcal{A}^*\right)\leq Ce^{-\min\{\frac{\sqrt{\lambda_1}}{2},\frac{l}{4}\} (t-t_{*}(B)-1)},\ \forall t\geq t_{*}(B)+1,
  \end{split}
\end{equation}
where~$t_{*}(B)$~is the entering time of~$B$~into~$\mathcal{B}_0$.
\end{theorem}

\begin{proof}
Write $\Psi(u_t(t,x))=\int_{\Omega}K(x,y)u_t(t,y)dy$. Let $w(t),v(t)$ be two weak solutions to $(\ref{wave equa3})$-$(\ref{initial condition3})$ corresponding to initial
data ~$y_1,y_2\in \mathcal{B}_0$, i.e.,
\begin{equation*}
(w(t),w_t(t))\equiv S(t)y_1, (v(t),v_t(t))\equiv S(t)y_2, \ y_1,y_2\in \mathcal{B}_0.
\end{equation*}
Since~$\mathcal{B}_0$~is positively invariant, we have
\begin{equation}\label{20-7-20-1}
\begin{cases}\|(w(t),w_t(t))\|_{ H^1_0(\Omega)\times L^2(\Omega)}\leq C,\\\|(v(t),v_t(t))\|_{ H^1_0(\Omega)\times L^2(\Omega)}\leq C,\end{cases}\forall t>0, y_1,y_2\in \mathcal{B}_0.
\end{equation}
The difference $z(t)=w(t)-v(t)$ satisfies
\begin{equation}\label{20-7-20-2}
z_{tt}-\Delta z +k(||w_t||^p w_t-||v_t||^p v_t)+lz_t+f(w)-f(v)=\Psi(z_t).
 \end{equation}
Let~$E(z)=\frac{1}{2}\big(\|z_t\|^2+\|\nabla z\|^2\big)$, $E_{\delta}(z)=E(z)+\delta (z_t,z)$. Take $\delta=\min\{\frac{\sqrt{\lambda_1}}{2},\frac{l}{4}\}$.
Then by Poincar\'e's inequality we have

\begin{equation}\label{20-7-20-3}
\delta|(z,z_t)|\leq \frac{\sqrt{\lambda_1}}{2} \frac{1}{\sqrt{\lambda_1}}\|\nabla z\| \|z_t\|\leq \frac{1}{4}\big( \|\nabla z\|^2+\|z_t\|^2\big)\leq \frac{1}{2} E(z)
\end{equation}
and
\begin{equation}\label{20-7-20-4}
 \frac{1}{2}E(z)\leq E_{\delta}(z)\leq \frac{3}{2}E(z).
\end{equation}

Multiplying $(\ref{20-7-20-2})$ by $z_t+\delta z$ in $L^2(\Omega)$, we obtain
\begin{equation}\label{20-7-20-5}
\begin{split}
  \frac{d}{dt}E_{\delta}(z)+2\delta E_{\delta}(z)=&-\big(k||w_{t}||^p w_t-k||v_{t}||^p v_t,w_t-v_t\big)-(l-2\delta)\|z_t\|^{2}\\
  &-\big(f(w)-f(v),w_t-v_t\big)+\big(\Psi(w_t-v_t),w_t-v_t\big)\\
  &-\delta\big(k||w_{t}||^p w_t-k||v_{t}||^p v_t,w-v\big)\\
  &-\delta\big(f(w)-f(v),w-v\big)+\delta\big(\Psi(w_t-v_t),w-v\big)\\
  &+(2\delta^2-\delta l)(z_t,z).
\end{split}
\end{equation}
It follows from~$(\ref{20-7-20-1})$~that
\begin{equation}\label{20-7-22-2}
  \|\Psi(w_t-v_t)\|\leq\|K\|_{L^2(\Omega\times\Omega)}\|w_t-v_t\|\leq C.
\end{equation}
We deduce from~$(\ref{growth})$~and~$(\ref{20-7-20-1})$ that
\begin{equation}\label{20-7-22-1}
\begin{split}
 &\|f(w)-f(v)\|\\=&\left\{\int_\Omega\bigg[\int_0^1f'\big(v+\theta(w-v)\big)\big(w-v\big)d\theta\bigg]^2dx\right\}^{\frac{1}{2}}\\
 \leq&C\left\{\int_\Omega\big(|w|^{\frac{4}{N-2}}+|v|^{\frac{4}{N-2}}+1\big)|w-v|^{2}dx\right\}^{\frac{1}{2}}\\
 \leq&C\big(\|w\|_{\frac{2N}{N-2}}^{\frac{2}{N-2}}+\|v\|_{\frac{2N}{N-2}}^{\frac{2}{N-2}}+1\big)\|w-v\|_{\frac{2N}{N-2}}\\
  \leq&C\big(\|\nabla w\|^{\frac{2}{N-2}}+\|\nabla v\|^{\frac{2}{N-2}}+1\big)\|\nabla(w-v)\|\\
  \leq &C.
\end{split}\end{equation}
We infer from~$(\ref{20-7-20-1})$,~$(\ref{20-7-22-2})$~and~$(\ref{20-7-22-1})$~that
\begin{equation}\label{20-7-20-6}
\begin{split}
  &-\delta\big(k||w_{t}||^p w_t-k||v_{t}||^p v_t,w-v\big)-\delta\big(f(w)-f(v),w-v\big)\\&+\delta\big(\Psi(w_t-v_t),w-v\big)+(2\delta^2-\delta l)(z_t,z)
  \leq C\|z\|
\end{split}
\end{equation}
and
\begin{equation}\label{20-7-20-8}
\begin{split}
  \big(\Psi(w_t-v_t),w_t-v_t\big)\leq C\|\Psi(w_t-v_t)\|.
\end{split}
\end{equation}
Plugging~$(\ref{20-7-20-6})$~and~$(\ref{20-7-20-8})$~into~$(\ref{20-7-20-5})$ gives
\begin{equation}\label{20-7-20-7}
\begin{split}
  &\frac{d}{dt}E_{\delta}(z)+2\delta E_{\delta}(z)\\\leq
  &-\big(f(w)-f(v),w_t-v_t\big)+C\|\Psi(w_t-v_t)\|+C\|z\|.
\end{split}
\end{equation}
By~Gronwall's inequality,~$(\ref{20-7-20-4})$ and~$(\ref{20-7-20-7})$, we have
\begin{equation}\label{20-7-20-9}
\begin{split}
 \frac{1}{2} E(z(t))\leq& E_{\delta}(z(t))\\\leq &e^{-2\delta t} E_{\delta}(z(0))+C\int_0^t \|z(s)\|ds\\
 &+C\int_0^t \|\Psi(w_t(s)-v_t(s))\|ds\\&+\left|\int_0^t e^{-2\delta(t-s)}\big(f(w(s))-f(v(s)),w_t(s)-v_t(s)\big)ds\right|\\
 \leq&e^{-2\delta t} E_{\delta}(z(0))+Ct\sup_{s\in[0,t]}\|z(s)\|\\
 &+C\int_0^t \|\Psi(w_t(s)-v_t(s))\|ds\\&+\left|\int_0^t e^{-2\delta(t-s)}\big(f(w(s))-f(v(s)),w_t(s)-v_t(s)\big)ds\right|.
\end{split}
\end{equation}
Write
\begin{equation}\label{20-7-21-1}
\begin{split}
\Phi_t(y_1,y_2)\equiv&Ct\sup_{s\in[0,t]}\|z(s)\|+C\int_0^t \|\Psi(w_t(s)-v_t(s))\|ds\\&+\left|\int_0^t e^{-2\delta(t-s)}\big(f(w(s))-f(v(s)),w_t(s)-v_t(s)\big)ds\right|.
\end{split}
\end{equation}
We derive from~$(\ref{20-7-20-4})$ and~$(\ref{20-7-20-9})$ that
\begin{equation}\label{20-7-20-10}
\begin{split}
\|S(t)y_1-S(t)y_2\|_{ H^1_0(\Omega)\times L^2(\Omega)}\leq &2e^{-\delta t}\sqrt{E_{\delta}(z(0))}+2\sqrt{\Phi_t(y_1,y_2)}\\
\leq&2e^{-\delta t}\sqrt{\frac{3}{2}E(z(0))}+2\sqrt{\Phi_t(y_1,y_2)}\\
\leq&Ce^{-\delta t}+2\sqrt{\Phi_t(y_1,y_2)}.
\end{split}
\end{equation}
Next, we will prove that
\begin{equation*}
  \liminf_{n\rightarrow\infty} \liminf_{m\rightarrow\infty}2\sqrt{\Phi_t(y^{(n)},y^{(m)})}=0, \ \forall \{y^{(n)}\}\subseteq \mathcal{B}_0,
\end{equation*}
which is equivalent to
\begin{equation}\label{20-7-20-20}
  \liminf_{n\rightarrow\infty} \liminf_{m\rightarrow\infty}\Phi_t(y^{(n)},y^{(m)})=0, \ \forall \{y^{(n)}\}\subseteq \mathcal{B}_0.
\end{equation}
We write~$u^{(n)}(t)=S(t)y^{(n)}$ and thereby have
\begin{equation*}
\begin{split}
  \Phi_t(y^{(n)},y^{(m)})=&\left|\int_0^t e^{-2\delta(t-s)}\big( f(u^{(n)}(s))-f(u^{(m)}(s)),u^{(n)}_t(s)-u^{(m)}_t(s)\big)ds\right|\\&+C \int_0^t\|\Psi(u^{(n)}_t(s)-u^{(m)}_t(s))\|ds+tC\sup_{s\in[0,t]}\|u^{(n)}(s)-u^{(m)}(s)\|.
\end{split}\end{equation*}
By Alaoglu's theorem, we deduce from
 \begin{equation}\label{20-7-23-1}
  \|(u^{(n)}(t),u^{(n)}_t(t))\|_{ H^1_0(\Omega)\times L^2(\Omega)}\leq C,\ \forall t\geq0
\end{equation}
that there exists a subsequence of~$\{u^{(n)}\}$~, still denoted by~$\{u^{(n)}\}$, such that
 \begin{equation}\label{20-7-20-30}
   (u^{(n)},u^{(n)}_t)\overset{\ast}\rightharpoonup (u,v)\ \  \text{in} \ L^{\infty}(0,t;H^1_0(\Omega)\times L^2(\Omega)).
 \end{equation}
Moreover, we can verify that $v=u_t$.

Indeed, by $(\ref{20-7-20-30})$, for any~$\phi(s)\in C_c^{\infty}[0,t]$~and any~$\psi_0(x)\in H^2(\Omega)\cap H^1_0(\Omega)$, we have
\begin{equation*}
\begin{split}
  &\int_0^t\big(u^{(n)}_t(s),\phi(s)\triangle\psi_0(x)\big)ds\\=&\int_0^t\phi(s)\frac{d}{dt}\big(u^{(n)}(s),\triangle\psi_0(x)\big)ds\\=&-\int_0^t\phi'(s)\big( u^{(n)}(s),\triangle\psi_0(x)\big)ds\\=&\int_0^t\big( \nabla u^{(n)}(s),\phi'(s)\nabla\psi_0(x)\big)ds\\
 \rightarrow &\int_0^t\big(\nabla u(s),\phi'(s)\nabla\psi_0(x)\big)ds\\ =&\int_0^t\big(u_t(s),\phi(s)\triangle\psi_0(x)\big)ds
  \end{split}
\end{equation*}
and
\begin{equation*}
  \int_0^t\big(u^{(n)}_t(s),\phi(s)\triangle\psi_0(x)\big)ds\rightarrow\int_0^t\big(v(s),\phi(s)\triangle\psi_0(x)\big)ds
\end{equation*}
as~$n\rightarrow\infty$. It follows that~$v=u_t$.

By~Arzel\`{a}-Ascoli~Theorem,
 \begin{equation}\label{20-7-20-11}
    C\big([0,T],H_0^1(\Omega)\big)\cap C^1\big([0,T],L^2(\Omega)\big)\hookrightarrow\hookrightarrow C\big([0,T],L^2(\Omega)\big).
 \end{equation}
Hence there exists a subsequence of~$\{u^{(n)}\}$~, still denoted by~$\{u^{(n)}\}$, such that
 \begin{equation}\label{20-7-20-33}
   u^{(n)}\rightarrow w\ \   \text{in}\ C([0,T];L^2(\Omega)).
 \end{equation}
 Besides, we have~$w=u$.

Indeed, by~$(\ref{20-7-20-30})$, for every~$\psi\in L^1(0,t;H^2(\Omega)\cap H^1_0(\Omega))$, we have
\begin{equation}\label{20-7-23-5}
  \left|\int_0^t\big(u^{(n)}(s)-u(s),\triangle\psi(s)\big)ds\right|=\left|\int_0^t\big(\nabla(u^{(n)}(s)-u(s)),\nabla\psi(s)\big)ds\right|\rightarrow 0.
\end{equation}
Meanwhile, it follows from~$(\ref{20-7-20-33})$~that
\begin{equation*}
\begin{split}
 \left|\int_0^t\big(u^{(n)}(s)-w(s),\triangle\psi\big)ds\right|\leq \sup_{s\in[0,t]}\|u^{(n)}(s)-w(s)\|\int_0^t\|\triangle \psi(s)\|ds\rightarrow0.
\end{split}
\end{equation*}
Consequently,~$w=u$.

Let~$\mathcal{A}$~be the strictly positive operator on~$L^2(\Omega)$~defined by~$\mathcal{A}=-\triangle$~with domain~$D(\mathcal{A})=H^2(\Omega)\cap H^1_0(\Omega)$.

 We derive from~$H^{N}(\Omega)\hookrightarrow L^{\infty}(\Omega)$~(where~$N$~ denotes the spatial dimension) that~$L^1(\Omega)\hookrightarrow (L^{\infty}(\Omega))^{*}\hookrightarrow H^{-N}(\Omega)$. Therefore, we deduce from~$(\ref{growth})$~and~$(\ref{20-7-23-1})$~that
\begin{equation}\label{20-7-23-7}
\begin{split}
  &\|\mathcal{A}^{-\frac{N}{2}}f(u^{(n)}(s))-\mathcal{A}^{-\frac{N}{2}}f(u(s))\|\\
= &\|f(u^{(n)}(s))-f(u(s))\|_{H^{-N}(\Omega)}\\
\leq &C\|f(u^{(n)}(s))-f(u(s))\|_{1}\\
=&C\int_{\Omega}\left|\int_0^1f'\big(u(s)+\theta(u^{(n)}(s)-u(s))\big)\cdot(u^{(n)}(s)-u(s))d\theta\right|dx\\
\leq &C\int_{\Omega}(|u^{(n)}(s)|^{\frac{2}{N-2}}+|u(s)|^{\frac{2}{N-2}}+1)\cdot|u^{(n)}(s)-u(s)|dx\\
\leq &C\|u^{(n)}(s)-u(s)\|\cdot(1+\|u^{(n)}(s)\|_{\frac{4}{N-2}}^{\frac{2}{N-2}}+\|u(s)\|_{\frac{4}{N-2}}^{\frac{2}{N-2}})\\
\leq &C\|u^{(n)}(s)-u(s)\|(1+\|\nabla u^{(n)}(s)\|^{\frac{2}{N-2}}+\|\nabla u(s)\|^{\frac{2}{N-2}})\\
\leq &C\|u^{(n)}(s)-u(s)\|
\end{split}
\end{equation}
holds for all~$s\in[0,t]$.

Combining~$(\ref{20-7-20-33})$~and~$(\ref{20-7-23-7})$~gives
\begin{equation*}
  \sup_{s\in[0,t]}\|\mathcal{A}^{-\frac{N}{2}}\big(f(u^{(n)}(s))-f(u(s))\big)\|\rightarrow0 \ \text{as}\ n\rightarrow\infty.
\end{equation*}
Hence
\begin{equation}\label{20-7-24-9}
\begin{split}
  &\left|\int_0^t\big(f(u^{(n)}(s))-f(u(s)),\varphi(s)\big)ds\right|\\
  =&\left|\int_0^t\Big(\mathcal{A}^{-\frac{N}{2}}\big(f(u^{(n)}(s))-f(u(s))\big),\mathcal{A}^{\frac{N}{2}}\varphi(s)\Big)ds\right|\\
  \leq &\sup_{s\in[0,t]}\left\|\mathcal{A}^{-\frac{N}{2}}\big(f(u^{(n)}(s))-f(u(s))\big)\right\|\int_0^t\|\varphi(s)\|_{H^N(\Omega)}ds\\
\rightarrow&0
\end{split}
\end{equation}
holds for all~$\varphi\in L^1(0,t;H^N(\Omega)\cap H^1_0(\Omega))$.
Since~$L^1\big(0,t;H^N(\Omega)\cap H^1_0(\Omega)\big)$~is dense in $L^1\big(0,t;L^{2}(\Omega)\big)$, $(\ref{20-7-24-9})$~implies
\begin{equation}\label{20-7-24-11}
f(u^{(n)})\overset{\ast}\rightharpoonup f(u)\ \   \text{in}\ L^{\infty}(0,t;L^2(\Omega)).
\end{equation}
We infer from~$(\ref{20-7-20-30})$~and~$(\ref{20-7-24-11})$~that
\begin{equation}\label{20-7-20-18}
\begin{split}
  &\lim_{n\rightarrow\infty}\lim_{m\rightarrow\infty}\int_0^t e^{-2\delta(t-s)}\big( f(u^{(n)}(s)),u^{(m)}_t(s)\big)ds\\
  =&\lim_{n\rightarrow\infty}\lim_{m\rightarrow\infty}\int_0^t \big( e^{-2\delta(t-s)}f(u^{(n)}(s)),u^{(m)}_t(s)\big)ds\\
   =&\lim_{n\rightarrow\infty}\int_0^t \big( e^{-2\delta(t-s)}f(u^{(n)}(s)),u_t(s)\big)ds\\
    =&\lim_{n\rightarrow\infty}\int_0^t \big( f(u^{(n)}(s)),e^{-2\delta(t-s)}u_t(s)\big)ds\\
     =&\int_0^t \big( f(u(s)),e^{-2\delta(t-s)}u_t(s)\big)ds\\
     =&\int_0^t e^{-2\delta(t-s)}\big( f(u(s)),u_t(s)\big)ds
\end{split}\end{equation}
and
\begin{equation}\label{20-7-20-19}
\begin{split}
  &\lim_{n\rightarrow\infty}\lim_{m\rightarrow\infty}\int_0^t e^{-2\delta(t-s)}\big( f(u^{(m)}(s)),u^{(n)}_t(s)\big)ds\\
     =&\int_0^t e^{-2\delta(t-s)}\big( f(u(s)),u_t(s)\big)ds.
\end{split}\end{equation}
Write~$F(\mu)=\int_0^{\mu}f(\tau)d\tau$. We deduce from~$(\ref{growth})$~and~$(\ref{20-7-23-1})$ that
\begin{equation}\begin{split}
 & \left|\int_{\Omega}F(u^{(n)}(s))dx-\int_{\Omega}F(u(s))dx\right|\\
  \leq&\int_{\Omega}\left|\int_0^1f\big(u(s)+\theta(u^{(n)}(s)-u(s))\big)\cdot(u^{(n)}(s)-u(s))d\theta\right|dx\\
\leq & C\int_{\Omega}(|u^{(n)}(s)|^{\frac{N}{N-2}}+|u(s)|^{\frac{N}{N-2}}+1)\cdot|u^{(n)}(s)-u(s)|dx\\
\leq &C\|u^{(n)}(s)-u(s)\|\cdot\big(1+\|u^{(n)}(s)\|_{\frac{2N}{N-2}}^{\frac{N}{N-2}}+\|u(s)\|_{\frac{2N}{N-2}}^{\frac{N}{N-2}}\big)\\
\leq &C\|u^{(n)}(s)-u(s)\|\big(1+\|\nabla u^{(n)}(s)\|^{\frac{N}{N-2}}+\|\nabla u(s)\|^{\frac{N}{N-2}}\big)\\
\leq &C\|u^{(n)}(s)-u(s)\|
\end{split}
\end{equation}
holds for all~$s\in[0,t]$, which together with~$(\ref{20-7-20-33})$, yields
\begin{equation}\label{20-7-20-17}
\sup_{s\in[0,t]}\left|\int_{\Omega}F(u^{(n)}(s))dx-\int_{\Omega}F(u(s))dx\right|\rightarrow 0\  (n\rightarrow\infty).
\end{equation}
It follows from~$(\ref{20-7-20-17})$~that
\begin{equation}\label{20-7-24-20}
\begin{split}
  &\lim_{n\rightarrow\infty}\int_0^t e^{-2\delta(t-s)}\big( f(u^{(n)}(s)),u^{(n)}_t(s)\big)ds\\
  =&\lim_{n\rightarrow\infty}\Big[\int_{\Omega}F(u^{(n)}(t))dx-e^{-2\delta t}\int_{\Omega}F(u^{(n)}(0))dx-\int_0^t 2\delta e^{-2\delta(t-s)}\int_{\Omega}F(u^{(n)}(s))dxds\Big]\\
  =&\int_{\Omega}F(u(t))dx-e^{-2\delta t}\int_{\Omega}F(u(0))dx-\int_0^t 2\delta e^{-2\delta(t-s)}\int_{\Omega}F(u(s))dxds\\
  =&\int_0^t e^{-2\delta(t-s)}\big( f(u(s)),u_t(s)\big)ds.
\end{split}\end{equation}
From~$(\ref{20-7-20-18})$,~$(\ref{20-7-20-19})$~and~$(\ref{20-7-24-20})$, we obtain
\begin{equation}\label{20-7-24-25}
\begin{split}
  &\lim_{n\rightarrow\infty}\lim_{m\rightarrow\infty}\left|\int_0^t e^{-2\delta(t-s)}\big( f(u^{(n)}(s))-f(u^{(m)}(s)),u^{(n)}_t(s)-u^{(m)}_t(s)\big)ds\right|\\
  =&\left|\lim_{n\rightarrow\infty}\lim_{m\rightarrow\infty}\Big[\int_0^t e^{-2\delta(t-s)}\big( f(u^{(n)}(s))-f(u^{(m)}(s)),u^{(n)}_t(s)-u^{(m)}_t(s)\big)ds\Big]\right|\\
  =&0.
\end{split}\end{equation}
Let $V$ be the completion of $L^2(\Omega)$ with respect to the norm $\|\cdot\|_{V}$ given by
$\|\cdot\|_V=\|\Psi(\cdot)\|+\|\mathcal{A}^{-\frac{1}{2}}\cdot\|$ and $W$ be the completion of $L^2(\Omega)$ with respect to the norm $\|\cdot\|_{W}$ given by $\|\cdot\|_W=\|\mathcal{A}^{-\frac{1}{2}}\cdot\|$. By Lemma $\ref{lemma 10-10-3}$, we have
\begin{equation}\label{20-7-20-12}
 L^2(\Omega)\hookrightarrow\hookrightarrow V\hookrightarrow W.
\end{equation}
By~$(\ref{20-7-22-1})$,
\begin{equation}\label{20-7-24-28}
\begin{split}
  \|f(u^{(n)}(t))\|\leq C.
\end{split}\end{equation}
In addition,
\begin{equation}\label{20-7-24-29}
\begin{split}
\|\Psi(u_{t}^{(n)}(t))\|\leq \|K\|_{L^2(\Omega\times\Omega)}\|u_{t}^{(n)}(t)\|\leq C.
\end{split}\end{equation}
We deduce from~$(\ref{wave equa3})$,~$(\ref{20-7-24-28})$~and~$(\ref{20-7-24-29})$~that
\begin{equation*}
\begin{split}
  \|\mathcal{A}^{-\frac{1}{2}}u_{tt}^{(n)}(s)\|\leq&\|\nabla u^{(n)}(s)\|+(k\|u_{t}^{(n)}(s)\|^p+l)\|\mathcal{A}^{-\frac{1}{2}}u_{t}^{(n)}(s)\|\\
  &+\|\mathcal{A}^{-\frac{1}{2}}\big(\Psi(u_{t}^{(n)}(s))+h-f(u^{(n)}(s))\big)\|\\\leq &C,
  \end{split}
\end{equation*}
which implies
\begin{equation}\label{20-7-20-13}
   \int_0^t\|\mathcal{A}^{-\frac{1}{2}}u_{tt}^{(n)}(s)\|ds\leq C_t.
\end{equation}
Besides, we have
\begin{equation}\label{20-7-20-14}
  \int_0^t\|u_{t}^{(n)}(s)\|dt\leq C_{t}.
\end{equation}
By Lemma $\ref{lemma 1-11-1}$,~$(\ref{20-7-20-12})$,~$(\ref{20-7-20-13})$ and~$(\ref{20-7-20-14})$  imply that $\big\{u^{(n)}_{t}(t)\big\}_{n=1}^{\infty}$ is relatively compact in $L^1(0,t;V)$. Thus there exists a subsequence of~$\big\{(u^{(n)},u^{(n)}_t)\big\}_{n=1}^{\infty}$~(still denoted by itself) such that
\begin{equation}\label{20-7-20-15}
 \lim_{n,m\rightarrow\infty}\int_0^t\|\Psi\big(u^{(n)}_{t}(s)-u^{(m)}_{t}(s)\big)\|ds=0.
\end{equation}
It follows from~$(\ref{20-7-20-33})$~that
\begin{equation}\label{20-7-20-22}
  \lim_{n,m\rightarrow\infty}\sup_{s\in[0,t]}\|u^{(n)}(s)-u^{(m)}(s)\|=0,
\end{equation}
which, together with~$(\ref{20-7-20-15})$~and Lemma $\ref{corollary 1-14-1}$, leads to
\begin{equation}\label{20-7-20-25}
  \liminf_{n\rightarrow\infty} \liminf_{m\rightarrow\infty} \big[C \int_0^t\|\Psi(u^{(n)}_t(s)-u^{(m)}_t(s))\|ds+tC\sup_{s\in[0,t]}\|u^{(n)}(s)-u^{(m)}(s)\|\big]=0.
\end{equation}
~$(\ref{20-7-20-20})$ follows from~$(\ref{20-7-24-25})$~and~$(\ref{20-7-20-25})$.

By Theorem~$\ref{20-7-24-35}$, combining~$(\ref{20-7-20-10})$ and~$(\ref{20-7-20-20})$ yields
\begin{equation}\label{21-8-29-1}
\begin{split}
\alpha(S(t)\mathcal{B}_0)\leq Ce^{-\delta t}=Ce^{-\min\{\frac{\sqrt{\lambda_1}}{2},\frac{l}{4}\} t},\ \forall t>0.
 \end{split}\end{equation}

Consequently, by Theorem~$\ref{20-8-5-3}$,~$(H_0^1(\Omega)\times L^2(\Omega),\{S(t)\}_{t\geq0})$~possesses a compact exponential attracting set~$\mathcal{A}^*$ such that for every bounded set~$B\subseteq X$ we have
\begin{equation}\label{21-8-29-3}
\begin{split}
\mathrm{ dist}\left(S(t)B, \mathcal{A}^*\right)\leq Ce^{-\min\{\frac{\sqrt{\lambda_1}}{2},\frac{l}{4}\} (t-t_{*}(B)-1)},\ \forall t\geq t_{*}(B)+1,
  \end{split}
\end{equation}
where~$t_{*}(B)$~is the entering time of~$B$~into~$\mathcal{B}_0$.

The proof is completed.
\end{proof}

As an application of Theorem~$\ref{22-1-28-1}$, the following theorem  establishes the estimation of the attractive velocity of a compact exponential attracting set for Problem $(\ref{wave equa3})$-$(\ref{initial condition3})$ by a method different from that in Theorem $\ref{20-8-2-4}$, and its estimation result is superior to the estimation obtained in Theorem $\ref{20-8-2-4}$.

\begin{theorem}\label{21-9-4-2}
 Under Assumption~$\ref{22-2-5-1}$, the dynamical system~$(H_0^1(\Omega)\times L^2(\Omega),\{S(t)\}_{t\geq0})$\ generated by problem~$(\ref{wave equa3})$-$(\ref{initial condition3})$ possesses a compact exponential attracting set~$\mathcal{A}^*$ such that for every bounded set~$B\subseteq X$ we have
\begin{equation}\label{22-1-29-56}
\begin{split}
\mathrm{ dist}\left(S(t)B, \mathcal{A}^*\right)\leq  2^{-\frac{l(t-t_{*}(B)-1)}{3}+2}\alpha\big(\mathcal{B}_0\big),\ \forall t\geq\frac{3}{l}+t_{*}(B)+1,
\end{split}
\end{equation}
where~$t_{*}(B)$ is the entering time of~$B$ into~$\mathcal{B}_0$.
\end{theorem}

\begin{proof}
Let~$w(t),v(t)$ be two weak solutions to~$(\ref{wave equa3})$-$(\ref{initial condition3})$ corresponding to initial
data~$y_1,y_2\in \mathcal{B}_0$.
The difference~$z(t)=w(t)-v(t)$ satisfies
\begin{equation}\label{21-9-2-53}
z_{tt}-\Delta z +k(||w_t||^p w_t-||v_t||^p v_t)+lz_t+f(w)-f(v)=\Psi(z_t).
 \end{equation}
Let~$E_z(t)=\frac{1}{2}(||z_t(t)||^2+||\nabla z(t)||^2)$ and~$T=\frac{3}{l}$. Multiplying~$(\ref{21-9-2-53})$ by~$z_t$~in~$L^2(\Omega)$ and integrating from~$t$ to~$T$, we obtain
 \begin{equation}\label{21-9-2-55}
\begin{split}
E_z(t)=&E_z(T)+k\int_t^T(||w_t||^p w_t-||v_t||^p v_t,z_t)d\tau+l\int_t^T||z_t||^2 d\tau+\int_t^T(f(w)-f(v),z_t)d\tau\\&-\int_t^T(\Psi(z_t),z_t)d\tau.
\end{split}
 \end{equation}
Integrating~$(\ref{21-9-2-55})$ from~$0$ to~$T$ yields
 \begin{equation}\label{21-9-2-56}
\begin{split}
TE_z(T)=&\int_0^TE_z(t)dt-\int_0^T\int_t^Tk(||w_t||^p w_t-||v_t||^p v_t,z_t)d\tau dt-l\int_0^T\int_t^T||z_t||^2 d\tau dt\\&-\int_0^T\int_t^T(f(w)-f(v),z_t)d\tau dt +\int_0^T\int_t^T(\Psi(z_t),z_t)d\tau dt.
\end{split}
 \end{equation}
Multiplying~$(\ref{21-9-2-53})$ by~$z$ in~$L^2(\Omega)$ and integrating from~$0$ to~$T$ yields
 \begin{equation}\label{21-9-2-58}
\begin{split}
\int_0^T E_z(t)dt=&-\frac{1}{2}(z_t,z)|^T_0+ \int_0^T ||z_t||^2 dt- \frac{1}{2} \int_0^T(f(w)-f(v),z) dt +  \frac{1}{2} \int_0^T(\Psi(z_t),z)  dt\\&- \frac{k}{2} \int_0^T(||w_t||^p w_t-||v_t||^p v_t,z)dt- \frac{l}{2} \int_0^T(z_t,z)dt.
\end{split}
\end{equation}
Substituting~$(\ref{21-9-2-58})$ into~$(\ref{21-9-2-56})$, we have
 \begin{equation}\label{20-7-25-7}
\begin{split}
TE_z(T)=&-\frac{1}{2}(z_t,z)|^T_0+ \int_0^T ||z_t||^2 dt- \frac{1}{2} \int_0^T(f(w)-f(v),z) dt +  \frac{1}{2} \int_0^T(\Psi(z_t),z)  dt
\\&- \frac{k}{2} \int_0^T(||w_t||^p w_t-||v_t||^p v_t,z)dt-\int_0^T\int_t^Tk(||w_t||^p w_t-||v_t||^p v_t,z_t)d\tau dt\\&-\int_0^T\int_t^T(f(w)-f(v),z_t)d\tau dt +\int_0^T\int_t^T(\Psi(z_t),z_t)d\tau dt- \frac{l}{2} \int_0^T(z_t,z)dt\\&-l\int_0^T\int_t^T||z_t||^2 d\tau dt.
\end{split}
 \end{equation}
 By Lemma 4.2 in \cite{my4}, we have
 \begin{equation*}
(||w_t||^p w_t-||v_t||^p v_t,w_t-v_t)\geq C_p||w_t-v_t||^{p+2}.
 \end{equation*}
Taking~$t=0$ in~$(\ref{21-9-2-55})$ yields
 \begin{equation}\label{20-7-25-15}
\begin{split}
\int_0^T ||z_t||^2 dt\leq &\int_0^T ||z_t||^2 dt+\frac{k}{l}\int_0^T(||w_t||^p w_t-||v_t||^p v_t,w_t-v_t)dt\\=&\frac{1}{l}\left(E_z(0)-E_z(T)-\int_0^T(f(w)-f(v),z_t)d\tau+\int_0^T(\Psi(z_t),z_t)d\tau\right).
\end{split}\end{equation}
It is easy to obtain the following estimate:
 \begin{equation}\label{21-12-7-3}
   -\frac{1}{2} \int_0^T(f(w)-f(v),z) dt \leq  TC\sup_{t\in[0,T]}\|z(t)\|,
\end{equation}
\begin{equation}\label{21-12-7-4}
  -\frac{1}{2}(z_t,z)|^T_0\leq C\sup_{t\in[0,T]}\|z(t)\|,
\end{equation}
\begin{equation}\label{21-12-7-5}
  \int_0^T(\Psi(z_t),z)  dt\leq TC\sup_{t\in[0,T]}\|z(t)\|,
\end{equation}
\begin{equation}\label{21-12-7-1}
  -\frac{k}{2} \int_0^T(||w_t||^p w_t-||v_t||^p v_t,z)dt\leq TC\sup_{t\in[0,T]}\|z(t)\| ,
\end{equation}
\begin{equation}\label{21-12-7-6}
  \int_0^T\int_t^T(\Psi(z_t),z_t)d\tau dt\leq TC\int_0^T\|\Psi(z_t(t))\|dt.
\end{equation}
Plugging~$(\ref{20-7-25-15})$-$(\ref{21-12-7-6})$ into~$(\ref{20-7-25-7})$, we obtain
 \begin{equation*}
\begin{split}
TE_z(T)\leq &C(T+1)\sup_{t\in[0,T]}\|z(t)\|+TC\int_0^T\|\Psi(z_t(t))\|dt+\left|\int_0^T\int_t^T(f(w)-f(v),z_t)d\tau dt\right|\\&+\frac{1}{l}\bigg(E_z(0)-E_z(T)+\left|\int_0^T(f(w)-f(v),z_t)d\tau\right|+C\int_0^T\|\Psi(z_t(t))\|dt\bigg),
\end{split}
 \end{equation*}
 i.e.,
 \begin{equation*}
\begin{split}
\Big(d\big(S(T)y_1,S(T)y_2\big)\Big)^2\leq & \frac{1}{1+lT}\bigg(\big(d(y_1,y_2)\big)^2+C_T\sup_{t\in[0,T]}\|z(t)\|+C_T\int_0^T\|\Psi(z_t(t))\|dt\\&+C_T\left|\int_0^T\int_t^T(f(w)-f(v),z_t)d\tau dt\right|+C_T\left|\int_0^T(f(w)-f(v),z_t)d\tau\right|\bigg),
\end{split}\end{equation*}
which implies
 \begin{equation}\label{22-1-28-51}
\begin{split}
d\big(S(T)y_1,S(T)y_2\big)\leq & \frac{1}{\sqrt{1+lT}}\bigg\{ d(y_1,y_2)+\Big(C_T\sup_{t\in[0,T]}\|z(t)\|+C_T\int_0^T\|\Psi(z_t(t))\|dt\Big)^{\frac{1}{2}}\\&+\left(C_T\left|\int_0^T\int_t^T(f(w)-f(v),z_t)d\tau dt\right|+C_T\left|\int_0^T(f(w)-f(v),z_t)d\tau\right|\right)^{\frac{1}{2}}\bigg\}.
\end{split}\end{equation}
Let~$\big(u^{(n)}(t),u^{(n)}_t(t)\big)=S(t)y^{(n)}~(\{y^{(n)}\}\subseteq \mathcal{B}_0)$,~$F(\mu)=\displaystyle\int_0^{\mu}f(\tau)d\tau$ and~$s\in (0,1)$. By Alaoglu's theorem and Lemma $\ref{lemma 1-11-1}$, we deduce from~$H^1_0(\Omega)\hookrightarrow\hookrightarrow H^s(\Omega)\hookrightarrow L^2(\Omega)$ that there exists a subsequence of~$\big\{(u^{(n)},u^{(n)}_t)\big\}_{n=1}^{\infty}$~(still denoted by~$\big\{(u^{(n)},u^{(n)}_t)\big\}_{n=1}^{\infty}$) such that
\begin{eqnarray}\label{1-13-1}
 \begin{cases}
 (u^{(n)},u^{(n)}_t)\overset{\ast}\rightharpoonup (u,u_t)\ \  \text{in} \ L^{\infty}(0,T;H^1_0(\Omega)\times L^2(\Omega)),\\
 u^{(n)}\rightarrow u\ \   \text{in }\ C([0,T];H^s(\Omega)),
 \end{cases}\  \ \text{as}\ n\rightarrow\infty.
\end{eqnarray}
By~$(\ref{growth})$, for all~$t\geq0$ we have
\begin{equation}\label{1-13-10}
\begin{split}
 \left|\int_{\Omega}F(u^{(n)}(t))dx-\int_{\Omega}F(u(t))dx\right|
\leq C\|u^{(n)}(t)-u(t)\|_{H^{s}(\Omega)}.
\end{split}
\end{equation}
Combining~$(\ref{1-13-1})$ and~$(\ref{1-13-10})$ gives
\begin{equation}\label{1-13-11}
  \int_{\Omega}F(u^{(n)}(t))dx\rightrightarrows\int_{\Omega}F(u(t))dx\ \text{as}\ n\rightarrow\infty.
\end{equation}
We deduce from~$(\ref{growth})$ that
\begin{equation}\label{1-13-4}
\begin{split}
  \|\mathcal{A}^{-\frac{N}{2}}f(u^{(n)}(t))-\mathcal{A}^{-\frac{N}{2}}f(u(t))\|
\leq C\|f(u^{(n)}(t))-f(u(t))\|_{1}
\leq C\|u^{(n)}(t)-u(t)\|_{H^{s}(\Omega)},~ \forall t\geq0.
\end{split}
\end{equation}
Combining~$(\ref{1-13-1})$ and~$(\ref{1-13-4})$ gives
\[\sup_{t\in[0,T]}\|\mathcal{A}^{-\frac{N}{2}}\big(f(u^{(n)}(t))-f(u(t))\big)\|\rightarrow0 \ \text{as}\ n\rightarrow\infty.\]
Consequently, for each fixed~$t\in[0,T]$ and each~$\varphi\in L^1\big(0,T;H^N(\Omega)\cap H^1_0(\Omega)\big)$, we have
~\[\int_t^T\big(f(u^{(n)}(\tau))-f(u(\tau)),\varphi\big)d\tau\rightarrow0\ \text{as}\ n\rightarrow\infty,\] which implies
   \begin{equation}\label{1-13-7}
f(u^{(n)})\overset{\ast}\rightharpoonup f(u)\ \text{in}\ L^{\infty}\big(t,T;L^2(\Omega)\big)\ \text{as}\ n\rightarrow\infty.
\end{equation}
From~$(\ref{1-13-1})$ and~$(\ref{1-13-7})$, we obtain
\begin{equation}\label{1-13-8}
\begin{split}
  \lim_{n\rightarrow\infty}\lim_{m\rightarrow\infty}\int_t^T\int_{\Omega}
  f(u^{(n)}(\tau))u^{(m)}_{t}(\tau)dxd\tau
  =\int_{\Omega}F(u(T))dx-\int_{\Omega}F(u(t))dx,\\
   \lim_{n\rightarrow\infty}\lim_{m\rightarrow\infty}\int_t^T\int_{\Omega}
  f(u^{(m)}(\tau))u^{(n)}_{t}(\tau)dxd\tau =\int_{\Omega}F(u(T))dx-\int_{\Omega}F(u(t))dx.
\end{split}\end{equation}
By Lebesgue's dominated convergence theorem, we deduce from~$(\ref{1-13-11})$ and~$(\ref{1-13-8})$ that
\begin{equation}\label{1-13-12}
\begin{split}
  &\lim_{n\rightarrow\infty}\lim_{m\rightarrow\infty}\left|\int_0^T\int_t^T\big(f(u^{(n)}(\tau))-f(u^{(m)}(\tau)),u^{(n)}_{t}(\tau)-u^{(m)}_{t}(\tau)\big)d\tau dt\right|\\=&\lim_{n\rightarrow\infty}\lim_{m\rightarrow\infty}\left|\int_0^T\big(f(u^{(n)}(t))-f(u^{(m)}(t)),u^{(n)}_{t}(t)-u^{(m)}_{t}(t)\big) dt\right|=0.
\end{split}\end{equation}
 As in the proof of Theorem~$\ref{20-8-2-4}$, we see that ~$\varrho_{1}(y_1,y_2)=\int_0^T\|\Psi(z_t(t))\|dt$, $\rho_{2}(y_1,y_2)=\sup_{t\in[0,T]}\|z(t)\|$ are precompact on~$\mathcal{B}_0$.

By Theorem~$\ref{22-1-28-1}$, we deduce from~$(\ref{22-1-28-51})$,~$(\ref{1-13-12})$ and the precompactness of $\varrho_{1}(y_1,y_2)=\int_0^T\|\Psi(z_t(t))\|dt$, $\rho_{2}(y_1,y_2)=\sup_{t\in[0,T]}\|z(t)\|$ that
\begin{equation}\label{21-12-21-7}
\begin{split}
\alpha\big(S(t)\mathcal{B}_0\big)\leq 2^{-\frac{lt}{3}+2}\alpha\big(\mathcal{B}_0\big),\ \forall t\geq\frac{3}{l}.
 \end{split}\end{equation}
Consequently, by Theorem~$\ref{20-8-5-3}$,~$(H_0^1(\Omega)\times L^2(\Omega),\{S(t)\}_{t\geq0})$~possesses a compact exponential attracting set~$\mathcal{A}^*$ satisfying ~$(\ref{22-1-29-56})$.
\end{proof}

 \section*{References }
\bibliographystyle{plain}

%% Authors are advised to submit their bibtex database files. They are
%% requested to list a bibtex style file in the manuscript if they do
%% not want to use model1a-num-names.bst.

%% References without bibTeX database:

% \begin{thebibliography}{00}
%
% \bibitem must have the following form:
%   \bibitem{key}...

 %\bibitem{}

\end{document}